\documentclass[12pt]{amsart}
\usepackage{amsmath,amssymb}
\usepackage{amscd}
\usepackage{layout}
\usepackage[mathscr,mathcal]{eucal}
\usepackage[pdftex]{graphicx}
 \newtheorem{thm}{Theorem}[subsection]
 
 \newtheorem{lem}[thm]{Lemma}
 \newtheorem{prop}[thm]{Proposition}
 \theoremstyle{definition}
 \newtheorem{defn}[thm]{Definition}
 \theoremstyle{remark}
 
 \numberwithin{equation}{subsection}

\begin{document}

\title[The Green's Functions of the Boundaries of the Hyperbolic 3-Manifolds]
{The Green's Functions of the Boundaries at Infinity of the
Hyperbolic 3-Manifolds}
\author{Majid Heydarpour}
\address{Majid Heydarpour, Department of Mathematics, Institute for Advanced Studies in Basic
Sciences, P.O. Box 45195-1159, Zanjan, Iran}

\email{Heydarpour@iasbs.ac.ir}

\date{}


\begin{abstract}
 The work is motivated by a result of Manin in \cite{Man1}, which
relates the Arakelov Green function on a compact Riemann surface to
configurations of geodesics in a 3-dimensional hyperbolic handlebody
with Schottky uniformization, having the Riemann surface as
conformal boundary at infinity.  A natural question is to what
extent the result of Manin can be generalized to cases where,
instead of dealing with a single Riemann surface, one has several
Riemann surfaces whose union is the boundary of a hyperbolic
3-manifold, uniformized no longer by a Schottky group, but by a
Fuchsian, quasi-Fuchsian, or more general Kleinian group. We have
considered this question in this  work and obtained several partial
results that contribute towards constructing an analog of Manin's
result in this more general context.

\end{abstract}
\maketitle

\section*{ Introduction}
\subsection*{Results} Lets consider that $S$ is a compact Riemann surface. Then we can
define a Green's function with respect to a metric and a divisor on
$S$. Suppose that $M$ is a hyperbolic 3 - Manifold with infinite
volume and having compact Riemann surfaces $S_1,S_2,...,S_n$ as its
conformal boundaries at infinity. Also choose a geometry on $M$
which bears a metric of constant negative curvature. The main new
results of this paper express the Green's functions on each $S_i$ in
terms of the length of some certain geodesics in $M$. Manin has done
this provided that $M$ has one boundary component (i.e. n=1) and is
uniformized by a Schottky group in \cite{Man1}. In this work we will
generalize Manin's results in general case, $M$ has more than one
boundary components and is uniformized by a Fuchsian, quasi-Fuchsian or Kleinian group.\\

In this introduction we state the main definitions, motivation and
the plan of doing the work. We start by introducing the definition
of a Green's function on a compact Riemann surface for a divisor and
with respect to a normalized volume form on it.

\subsection*{Green's functions on Riemann surfaces} Let $S$ be a compact
Riemann surface and $A=\sum_x m_x(x)$ ( with $m_x$ integer number )
be a divisor on $S$. we show the support of $A$ by $|A|$. Also lets
consider that $d\mu$ is  a positive real - analytic 2 - form  on
$S$. By a Green's function on $S$ for $A$ with respect to $d\mu$ we
mean a real analytic function
$$g_{\mu,A}=g_A:S\setminus |A|\longrightarrow\mathbb{R}$$
satisfying the following conditions:\\
(i) \emph{Laplace equation: }
$$\partial\bar\partial g_A=\pi i(deg(A)d\mu - \delta_A ).$$
Where $deg(A)=\sum_x m_x$, and  $\delta_A$ is the standard
$\delta$ - current $\varphi\mapsto \sum_x m_x\varphi(x)$. \\

(ii) \emph{Singularities:} Let $z$ be a complex coordinate in a
neighborhood of the point $x$. Then $g_A - m_x\log |z|$ is locally
real analytic.\\

A function satisfying these two conditions is uniquely determined up
to an additive constant. And the third condition is

(iii) \emph{Normalization:}
$$\int_Sg_Ad\mu=0.$$
Which eliminates the remaining ambiguous constant. $g_A$ is additive
on $A$ and for $x\neq y$, $g_x(y)$ is symmetric, i.e.
$g_x(y)=g_y(x)$. \\

Lets consider that $B=\sum_yn_y(y)$ is another divisor on $S$ that
is prime to the divisor $A$. That means, $|A|\cap |B|=\emptyset$.
And Put
\begin{align}
g_\mu(A, B):=\sum_yn_yg_{\mu,A}(y).\label{F0.10}
\end{align}
$g_A$ is additive and $g_x(y)$ is
symmetric, then $g_\mu(A, B)$ is symmetric and biadditive in $A,B$.\\

In general, the function $g_{\mu}$ depends on the metric $d\mu$, but
in the case that both of the divisors are of the degree zero, from
the condition (i) we see that  $g_{\mu, A, B}$ depends only on $A,
B$. Notice that, as a particular case of the general Kahler
formalism, to choose $d\mu$ is the same as to choose a real analytic
Riemannian metric on $S$ compatible with the complex structure. This
means that $g_{\mu}(A, B)=g(A, B)$ are conformal invariants when
both divisors are of degree zero. Also in the case that the divisor
$A$ is principal, i.e. $A$ is the divisor of a meromorphic function
like $f_A$, then
\begin{align}
g(A, B)=\log \prod_{y\in
|B|}|f_A(y)|^{n_y}=Re\int_{\gamma_B}\frac{df_A}{f_A}\label{F0.11}
\end{align}
Where the curve $\gamma_B$ is a 1 - chain with boundary $B$. The
divisors of degree zero on the Riemann sphere
$\mathbb{P}^1(\mathbb{C})$ are principal. Then this formula can be
directly applied to divisors of degree zero on Riemann sphere.\\

It is  well known that the Green's function of the degree zero
divisors on a Riemann surface of arbitrary genus can be expressed
exactly via the differential of the third kind $\omega_A$ with pure
imaginary periods and residues $m_x$ at $x$ when the divisor is
$A=\sum m_x(x)$(see, \cite{AG}, \cite{IA}). Then the generalization
of the previous  formula for arbitrary divisors $A,B$ of degree zero
is
\begin{align}
g(A, B)=Re \int_{\gamma_B}\omega_A.\label{F0.12}
\end{align}

In general, when the  degree of the divisors are not restricted to
zero, the basic Green's function $g_{\mu, (x)}(y)$ can be expressed
explicitly via  theta functions ( as in \cite{Fal} ) in the case
when $\mu$ is the \emph{Arakelov metric} constructed with the help
of an orthonormal basis of the differentials of the first kind on
the Riemann surface.

\subsection*{Motivations}  The result of Manin in \cite{Man1} which
relates the Arakelov Green function on a compact Riemann surface to
configurations of geodesics in a 3-dimensional hyperbolic handlebody
with Schottky uniformization, having the Riemann surface as
conformal boundary at infinity, was extremely innovative and
influential and had a wide range of consequences in the arithmetic
context of Arakelov geometry as well as and in other contexts,
ranging from p-adic geometry, real hyperbolic 3-manifolds, the
holography principle and AdS/CFT correspondence in string theory,
and noncommutative geometry. A natural question is to what extent
the result of Manin can be generalized to cases where, instead of
dealing with a single Riemann surface, one has several Riemann
surfaces whose union is the boundary of a hyperbolic 3-manifold,
uniformized no longer by a Schottky group, but by a Fuchsian,
quasi-Fuchsian, or more general Kleinian group. Such a
generalization is not only interesting because it is a very natural
question to pass from Kleinian-Schottky groups to more general
Kleinian groups, but also for its potential applications to Arakelov
geometry, to the case of curves defined over number fields with
several Archimedean places, while Manin's result was formulated for
the case of arithmetic curves defined over the rationales.

\subsection*{Plan} We have focused on the formula in Manin's work,
that expresses the Arakelov Green's function on a compact Riemann
surface in terms of a basis of holomorphic differentials of the
first kind and of differentials of the third kind. In the case of
Schottky uniformization, when the limit set has Hausdorff dimension
strictly smaller than one, one can construct such differentials in
terms of averages over the Schottky group. While the same type of
formula no longer holds in the Fuchsian or quasi-Fuchsian case, we
use the canonical covering map relating Fuchsian and Schottky
uniformization and the coding of limit sets for the Fuchsian and
Schottky case, to express the Green function in the Fuchsian or
quasi-Fuchsian case in terms of the one in the Schottky case. The
approach we follow for the more general Kleinian case is via a
decomposition of the uniformizing group as a free product of
quasi-Fuchsian and Schottky groups and applying the results we
obtained for these cases individually. The paper consists of three
chapters devoted to the cases that the hyperbolic 3-manifold $M$
have 1, 2 and $n<\infty$ many boundary components at infinity
respectively. First chapter pays to the one boundary component case
and includes four subchapter devoted to the Foundations and the
genera 0,1 and $\geq 2$ respectively. Also this chapter contains
fundamental definitions and basic computations that are bases for
the computations in the next chapters. The results of this chapter
are a summary of the reference \cite{Man1}. All hyperbolic
3-manifolds with two boundary components with the same genera are
uniformized by Quasi-fuchsian groups; the second chapter is devoted
to this manifolds. In third chapter we consider the general case
i.e. the case that $M$ is uniformized by some Kleinian groups and
have $n<\infty$ boundary components. Finally, as a remark for the
chapter three we show that for manifolds with  $\infty$ many
boundary components at infinity the situation is like the previous.

\subsection*{Acknowledgement} I am grateful to Matilde Marcolli who
suggested this problem, supported me in Max-Planck and Hausdorff
institute in Bonn, and supervised me in all process. I also many
thank Saad Varsaie for several useful and encouraging comments.

\section{One boundary component case}
In this chapter, first we gather some fundamental definitions and
some useful notations that will be used in all of the paper. Next we
bring the computation of Green's function for Riemann sphere
$\mathbb{P}^1(\mathbb{C})$, the boundary at infinity of the
hyperbolic space. All 3-manifolds with one boundary component that
is a compact Riemann surface with genus $\geq 1$, are uniformized by
Schottky groups. The rest of the chapter is devoted to these
manifolds. The results of this chapter are a summary of the
reference \cite{Man1}.

\subsection{Foundations} Consider the hyperbolic space $H^3$ with
the upper half space model and coordinate $(z,y)$ that comes from
$\mathbb{C}\times{\mathbb{R}}_+$ equipped by the hyperbolic distance
function corresponding to the metric
\begin{align}
ds^2=\frac{|dz|^2+dy^2}{y^2}\label{F1.1}
\end{align}
of constant curvature -1. The geodesics in this model are vertical
half - lines $Z=constant$ and also vertical half - circles
orthogonal to the plane at infinity $\mathbb{C}$ ( i.e. for $y=0$ ).\\

If we consider the end points of the geodesics ( including $\infty$
for the ends of the vertical half-lines), then we can consider the
Riemann sphere $\hat{\mathbb{C}}=\mathbb{C}\cup \{\infty\}$ ( or
$S^2$ in the unit boll model ) as the boundary at infinity of $H^3$.
\subsubsection{Notations} Lets show the geodesic joining $a$ to $b$ in
$H^3\cup \hat{\mathbb{C}}$ by $\{a, b\}$; by $a\ast\gamma$ the point
on the geodesic $\gamma$ closest to the point $a$ ( i.e, the
intersection point of $\gamma$ and a geodesic passing through $a$
and orthogonal to $\gamma$ ); by $d_u(a, b)$ the distance from $u$
to the geodesic $\{a, b\}$ and by $ordist(a, b)$ the oriented
distance between two points lying on an oriented geodesic in $H^3$ (
Figure \ref{Fig1.10}). Also lets show by $\varphi_u(a, b)$ the angle
(at $u$ ) between the semi - geodesics joining $u$ to $a$ and $u$ to
$b$, and for $a,b$ in $\hat{\mathbb{C}}$ by $\psi_{\gamma}(a, b)$
the oriented angle between the semi - geodesics joining
$a\ast\gamma$ to $a$ and $b\ast\gamma$ to $b$. In order to calculate
$\psi_{\gamma}(a, b)$ we must first make the parallel translation
along $\gamma$ identifying the normal spaces to $\gamma$ at
$a\ast\gamma$ and at $b\ast\gamma$. The orientation of a normal
space is defined by projecting it along oriented $\gamma$ to its
initial end into the tangent space to $\hat{\mathbb{C}}$, which is
canonically oriented by the complex structure.
\begin{figure}[!hbtp]
\centerline{\includegraphics{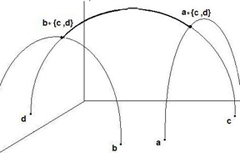}} \caption { }\label{Fig1.10}
\end{figure}
\subsection{Genus zero case} In this case we consider that the hyperbolic manifold $M$ is
the hyperbolic space $H^3$ with Riemann-sphere as its boundary at
infinity. For this case we have
\begin{prop}\label{P1.5} For $a, b, c$ and $d$ in
$\mathbb{P}^1(\mathbb{C})$, denote by $w_{(a)-(b)}$ a meromorphic
function on $\mathbb{P}^1(\mathbb{C})$ with the divisor $(a)-(b)$.
Then we have
\begin{align}
\log\biggm|\frac{w_{(a)-(b)}(c)}{w_{(a)-(b)}(d)}\biggm|= -
ordist(a\ast\{c, d\},b\ast\{c, d\})\label{F1.3}
\end{align}
and
\begin{align}
\arg\frac{w_{(a)-(b)}(c)}{w_{(a)-(b)}(d)}= - \psi_{\{c,d\}}(a,
b)\label{F1.4}
\end{align}
\end{prop}
\begin{proof} Mobius transformations preserve hyperbolic distance and angel. Then both
sides of (\ref{F1.3}) and (\ref{F1.4}) are invariant under these
transformations. Hence it suffices to consider the case when
$(a,b,c,d)=(z,1,0,\infty)$ in $\mathbb{P}^1(\mathbb{C})$. Then the
geodesic $\{a,b\}=\{0,\infty\}$ is the $y$ semi-axis in
$(z,y)=(z_1+iz_2,y)$ coordinate for $H^3$ and the geodesics joining
the points $a=z$ and $b=1$ normally to this semi-axis are half
circles passing from these points with center in $0$ and normal to
$\mathbb{C}$ ( Figure \ref{Fig1.11}).
\begin{figure}[!hbtp]
\centerline{\includegraphics{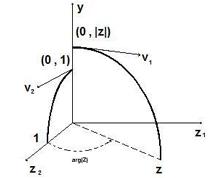}} \caption { }\label{Fig1.11}
\end{figure}
Then we have
$a\ast\{c,d\}=z\ast\{0,\infty\}=(0,|z|)~~,~~b\ast\{c,d\}=1\ast\{0,\infty\}=(0,1)$
and
$$ordist(a\ast\{c, d\},b\ast\{c, d\})=ordist((0,|z|), (0,1))=-\log|z|.$$
Also from the properties of cross - ratio we have
$$\frac{w_{(a)-(b)}(c)}{w_{(a)-(b)}(d)}=\frac{w_{(z)-(1)}(0)}{w_{(z)-(1)}(\infty)}=z.$$
This gives (\ref{F1.3}) directly. To see (\ref{F1.4}), we know that
angles in $H^3$ can be calculated using the Euclidean metric
$|dz|^2+dy^2$ and also the parallel transport along $y$ semi-axis
coincides with the Euclidean one. As it's shown in figure
\ref{Fig1.11} the vectors $v_1$ and $v_2$ tangent to the geodesics
from the points $z$ and $1$, are normal to $\{0,\infty\}$. Hence  if
we  transport (parallel) the vector $v_2$ in the point $(0,1)$ to
the point $(0,|z|)$ then both vectors are in an Euclidean plane
normal to $y$ semi-axis. Then the angel from  the vector $v_1$ to
the vector $v_2$ can be calculated via there angels from $z_1$ axis.
Then we have
$$\psi_{\{0,\infty\}}(z,1)=arg(v_2)-arg(v_1)=-arg\frac{w_{(z)-(1)}(0)}{w_{(z)-(1)}(\infty)}.$$
This proves (\ref{F1.4}).
\end{proof}
The part $w_{(a)-(b)}(c)/w_{(a)-(b)}(d)$ in the proposition above is
the classical cross - ratio of four points on the Riemann sphere
$\mathbb{P}^1(\mathbb{C})$, for which it is convenient to have a
special notation
\begin{align}
\langle a, b, c, d\rangle:=\frac{w_{(a)-(b)}(c)}{w_{(a)-(b)}(d)}.\label{F1.5}
\end{align}
\begin{thm} Let $a, b, c$ and $d$ be in
$\mathbb{P}^1(\mathbb{C})$. Then we have\\
\begin{align}
g((a)-(b),(c)-(d))&= - ordist(a\ast\{c,d\},b\ast\{c,d\})\nonumber\\
&=\log|\langle a, b, c, d\rangle|.\label{F1.6}
\end{align}\\
\end{thm}
\begin{proof} From the formula (\ref{F0.11}) for the Green's
function in the case that the divisors  are principal, we have\\
\begin{align}
g((a)-(b),(c)-(d))&=\log|w_{(a)-(b)}(c)|^1\log|w_{(a)-(b)}(d)|^{-1}\nonumber\\
&=\log\biggm|\frac{w_{(a)-(b)}(c)}{w_{(a)-(b)}(d)}\biggm|.\nonumber
\end{align}\\
Then from proposition (\ref{P1.5}) and notation (\ref{F1.5}) we have
(\ref{F1.6}).\end{proof}
\subsection{Genus one case}  In this case the group that uniformize the
hyperbolic 3 - manifold  with its boundary at infinity is a cyclic
group, and there is a nice explicit formula for the basic Green's
function on the boundary at infinity of 3-manifold; but because we
do not use it here we do not want to point to it here. But one can
find it in \cite{Man1}( or for a physical point of view in
\cite{MM}). Of course the process in the next part can be used for
this case too. We have pointed some notes about this case in section
\ref{S1.5}.
\subsection{Genus$>$1 case and Schottky Groups} Consider the complex
projective linear transformations group
$\mathrm{PGL}(2,\mathbb{C})$. \emph{(i) Kleinian groups.} A subgroup
$G$ of $\mathrm{PGL}(2,\mathbb{C})$ is called a Kleinian group if
$G$ acts on $H^3$ properly discontinuously. For a Kleinian group $G$
and a point $p$ in $H^3$ lets denote by $G(p)$, the orbit of the
point $P$ under the action of $G$. Since $G$ acts on $H^3$ Properly
discontinuously, $G(p)$ has accumulation points only on
$\mathbb{P}^1(\mathbb{C})$. They are independent of the choice of
the reference point $p$ and are called the \emph{limit set} of $G$,
which is denoted by $\Lambda(G)$. Equivalently, it is the closure of
the set of all fixed points of the elements of $G$ other than the
identity. $\Lambda(G)$ is the minimal non - empty closed $G$ -
invariant set. The complement of the limit set
$\mathbb{P}^1(\mathbb{C})\setminus\Lambda(G)$ is denoted by
$\Omega(G)$ and is called the \emph{region of discontinuity} of
$G$.\\

 The Kleinian group $G$ is called of the \emph{first kind} if
$\Lambda(G)=\mathbb{P}^1(\mathbb{C})$ and of the \emph{second kind}
if $\Omega(G)\neq\emptyset$. In the case that $G$ is of the second
kind $G$ acts on $\Omega(G)$ properly discontinuously. Therefor in
this case $\Omega(G)$ is the maximal open $G$ - invariant subset of
$\mathbb{P}^1(\mathbb{C})$ where $G$ acts properly discontinuously.
The quotient space $\Omega(G)/G$ has the complex structure induced
from that of $\Omega(G)$. Thus $\Omega(G)/G$ is a countable union of
Riemann surfaces lying at infinity of the complete hyperbolic 3 -
manifold $H^3/G$. For a torsion free Kleinian group $G$, a manifold
$(H^3\cup\Omega(G))/G$ possibly with boundary is denoted by $M_G$
and is called a \emph{Kleinian manifold}. The interior $H^3/G$ of
$M_G$ which admits the hyperbolic structure is denoted by $N_G$.\\

 \emph{(ii) Loxodromic, Parabolic and Elliptic elements.} An element
$g$ in $\mathrm{PGL}(2,\mathbb{C})$ is called loxodromic if it has
exactly two different fixed points in $\mathbb{P}^1(\mathbb{C})$.
These points are denoted by $z^+(g)$ and $z^-(g)$ and are called
attracting one and repelling one. For any $z_0\neq z^\pm(g)$ and
$h\in\mathrm{PGL}(2,\mathbb{C})$, we have
$z^\pm(g)=\lim_{n\rightarrow\pm\infty}g^n,~z^\pm(g)=z^\mp(g^{-1}),~z^\pm(hgh^{-1})=hz^\pm(g)$.
If we denote by $q(g)$ the eigenvalue of $g$ on the complex tangent
space to $z^+(g)$  then there is a local coordinate for
$\mathbb{P}^1(\mathbb{C})$ that in this coordinate $g$ is
represented by
$\left(\begin{array}{cc}q(g)&0\\0&1\end{array}\right)$. For this
reason  $q(g)$ is called the multiplier of $g$, and we have
$|q(g)|<1$, $q(g)=q(g^{-1})=q(hgh^{-1})$. Also by definition $g$ is
called a parabolic element if it has precisely one fixed point in
$\mathbb{P}^1(\mathbb{C})$, and elliptic if it has fixed points in
$H^3$. In fact an elliptic element fixes all points of a geodesic in
$H^3$ joining two fixed points in $\mathbb{P}^1(\mathbb{C})$. For a
loxodromic or elliptic element $g$ the geodesic joining two fixed
points in $\mathbb{P}^1(\mathbb{C})$ is called axis of $g$. An
element is elliptic if and only if it has finite order, then
elliptic elements cause a singularity in $N_G$. For this reason we
consider Kleinian groups without elliptic elements. This means
that the group is torsion free or equivalently acts freely on $H^3$.\\

\emph{(iii) Schottky groups}. A finitely generated, free and purely
loxodromic  Kleinian group is called a Schottky group. Purely
loxodromic means,  all elements except the identity are loxodromic.
For such a group $\Gamma$ the number of a minimal set of generators
is called the genus of $\Gamma$. A marking for a Schottky group
$\Gamma$ of genus $p$ by definition is a family of $2p$ open
connected domains $D_1,...,D_{2p}$ in $\mathbb{P}^1(\mathbb{C})$ and
a family of generators $g_1,...,g_p\in\Gamma$ with the following
properties. For $i=1,...,p$\\

a) The boundary $C_i$ of $D_i$ is a Jordan curve homeomorphic to
$S^1$ and  closures of $D_i$ are pairwise disjoint.

b) $g_i(C_i)\subseteq C_{p+i}$ and
$g_i(D_i)\subset\mathbb{P}^1(\mathbb{C})\setminus D_{p+i}$.\\

 We say that a marking is classical, if all $C_i$ are circles. It is known
that every Schottky group admits a marking. In fact each Schottky
group admit infinitely many marking but there are some Schottky
groups for which no classical marking exists.\\

\emph{(iv) $\Gamma$ - invariant sets of  the Schottky groups and
there quotient spaces}. A Schottky group $\Gamma$ is a Kleinian
group, then it acts on $H^3$ properly discontinuously. This action
is free too. Then the  quotient space $N_{\Gamma}=H^3/\Gamma$ has a
complete  hyperbolic 3-manifold structure this means that it is a
non-compact Riemann space of constant curvature -1. Topologically,
If $\Gamma$ is of genus $p$ then $N_{\Gamma}=H^3/\Gamma$
is the interior of  a handlebody of genus $p$.\\

Now lets consider  the  marking $\{~ D_1,...,D_{2p}~;~
g_1,...,g_p~\}$ for the Schottky group $\Gamma$ of genus $p$ and put
\begin{align}
X_{\Gamma}:=\mathbb{P}^1(\mathbb{C})\setminus\bigcup_{i=1}^{p}(D_i\cup
{\bar{D}}_{p+i}),\hspace{5mm}
\Omega(\Gamma):=\bigcup_{g\in\Gamma}g(X_{\Gamma}).\label{F1.9}
\end{align}
The set $\Omega(\Gamma)$ is the region of discontinuity of $\Gamma$
and $X_{\Gamma}$ is a fundamental domain for the action of $\Gamma$
on $\mathbb{P}^1(\mathbb{C})$. Then $\Gamma$  acts on
$\Omega(\Gamma)$ freely and properly discontinuously. So the
quotient space $S_{\Gamma}=\Omega(\Gamma)/\Gamma$ is a complex
Riemann surface of genus $p$. All compact Riemann surfaces can be
obtained in this way ( see \cite{H}) and every compact Riemann
surface admits infinitely many different Schottky covers.\\

The Cayley graph of a Schottky group $\Gamma$ of genus $p$ is an
infinite tree with multiplicity of $2p$ at each vertex. $\Gamma$ is
free, then this tree is without loops and each path between two
points is unique. Now, by the definition of a marking for $\Gamma$,
each generator $g_i$ takes $X_\Gamma$ to the inside of $C_{p+i}$ and
again the generator $g_j$ takes this part to the inside of the
second copy of $C_{p+j}$ inside $C_{p+i}$. This means that for each
element $g$ in $\Gamma$ we can associate a path in the Cayley graph
that the points of $X_\Gamma$ moves along that path on a tubular
neighborhood. This  express the action of $\Gamma$ on the set
$\Omega(\Gamma)$ ( Figure \ref{Fig1.12} illustrating this for the
case $p=2$ ).
\begin{figure}[!hbtp]
\centerline{\includegraphics{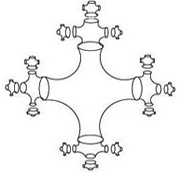}} \caption { }\label{Fig1.12}
\end{figure}
 We can consider the Riemann surface $S_{\Gamma}$ as the boundary at infinity of the
manifold $N_{\Gamma}$ by identifying the points of $S_{\Gamma}$ with
the points of the set of equivalence classes of unbounded ends of
oriented geodesics in $N_{\Gamma}$ modulo the relation "distance=0". \\

By the definition of the limit set for a Kleinian group, the
complement
\begin{align}
\Lambda(\Gamma):=\mathbb{P}^1(\mathbb{C})\setminus\Omega(\Gamma)
\end{align}
is the limit set of $\Gamma$. For the cyclic Schottky groups ( of
genus 1 ) the limit set  $\Lambda(\Gamma)$ consists of two points
which can be chosen as $0, \infty $, but for genus$\geq2$ this set
is an infinite Cantor set ( see \cite{FH}).\\

Lets consider the irreducible left - infinite words like
$\bar{h}=...h_i^{\varepsilon_i}...h_0^{\varepsilon_0}$. Where
$h_i\in\{g_1,...,g_p\}$,  $\varepsilon_i=\pm1$, and  $z_0$ is a
point in the fundamental domain $X_{\Gamma}$. And put
\begin{align}
z^+(\bar{h})=\lim_{i\rightarrow
+\infty}h_i^{\varepsilon_i}...h_0^{\varepsilon_0}(z_0)
\end{align}
This is a well defined point of $\Lambda(\Gamma)$ and is independent
of the point $z_0$. Since $\Gamma$ is a free group, the map
$\bar{h}\mapsto z^+(\bar{h})$ establishes a bijection as following
\begin{center}
Irreducible left - infinite words in $g_i$ \\
$\updownarrow$\\
Ends of the Cayley graph of $\Gamma$, $\{g_i\}$ \\
$\updownarrow$\\
 Points of $\Lambda(\Gamma).$\\
\end{center}
Denote by $a(\Gamma)$ the Hausdorff dimension of the set
$\Lambda(\Gamma)$. It can be characterized as the convergence
abscissa of any Poincare series
\begin{align}
\sum_{g\in\Gamma}\biggm|\frac{dg(z)}{dz}\biggm|^s.
\end{align}
Where $z$ is any coordinate function on $\mathbb{P}^1(\mathbb{C})$
with a zero and a pole in $\Omega(\Gamma)$. For $s\in\mathbb{C}$
that Re$(s)>a(\Gamma)$, this series converges uniformly on compact
subsets of $\Omega(\Gamma)$. Generally we have $0<a(\Gamma)<2$ (See
\cite{B}). In the next section we will consider $a(\Gamma)<1$ to
have the convergency of some infinite products defining some
$\Gamma$ - automorphic functions. This class of Schottky groups was
characterized by Bowen \cite{B} in following way: $a(\Gamma)<1$ if
and only if $\Gamma$ admits a rectifiable invariant quasi-circle
(which then contains $\Lambda(\Gamma)$ ).\\

Choose marking for $\Gamma$ of genus $p$ and denote by $a_i$ the
image of $C_i$ in $S_\Gamma$ with induced orientation. Also for
$i=1,...,p$ choose the points $x_i$ in $C_i$ and denote by $b_i$ the
images in $S_\Gamma$ of oriented pathes  from $x_i$ to $g_i(x_i)$
lying in $X_{\Gamma}$. These images are obviously closed paths and
we can choose them in such a way that they don't intersect. If we
denote the classes of these pathes in 1- homology group
$H_1(S_{\Gamma},\mathbb{Z})$ by the same notations, then the set
$\{a_i, b_j\}$ form a canonical basis of this group i.e. for all
$i=1,...,p$ we have
\begin{align}
(a_i,a_j)=(b_i, b_j)=0~,~ (a_i,b_j)=\delta_{ij}.
\end{align}
Moreover the kernel of the map
$H_1(S_{\Gamma},\mathbb{Z})\longrightarrow
H_1(M_{\Gamma},\mathbb{Z})$ which is induced by the inclusion
$S_\Gamma \hookrightarrow M_\Gamma$ is generated by the classes
$a_i$.
\subsubsection{Differentials of the first kind}\label{S1.5} Lets consider the cyclic
Schottky group $\Gamma=<qz>$, for $q\in\mathbb{C}^*$, $|q|<1$. In
this case $\Lambda(\Gamma)=\{0,\infty\}$ and consequently
$\Omega(\Gamma)=\mathbb{C}^*$. Then a differential of the first kind
on $\Omega(\Gamma)=\mathbb{C}^*$ can be written as
\begin{align}
\omega=d\log z =d\log
\frac{w_{(0)-(\infty)}(z)}{w_{(0)-(\infty)}(z_0)}=d\log \langle0,
\infty, z, z_0 \rangle.
\end{align}
where $z_0$ is any point $\neq 0, \infty$. And for $d\log q^nz
=d\log z$, this differential determines a differential of the first
kind on $S_\Gamma$.  In general case for a Schottky group $\Gamma$
of genus $p$ we can make a differential of the first kind $\omega_g$
for any $g\in\Gamma$ on $\Omega(\Gamma)$ and $S_\Gamma$ by an
appropriate averaging of this formula. Lets consider a marking for
$\Gamma$ and Denote by $C(|g)$ a set of representatives of
$\Gamma/(g^{\mathbb{Z}})$; by $C(h|g)$ a similar set for
$(h^{\mathbb{Z}})\setminus\Gamma/(g^{\mathbb{Z}})$; and by $S(g)$
the conjugacy class of $g$ in $\Gamma$. Then for any
$z_0\in\Omega(\Gamma)$ we have
\begin{prop} (a) If $a(\Gamma)< 1$, the following series converges
absolutely for $z\in\Omega(\Gamma)$ and determines (the lift to
$S_{\Gamma}$ of ) a differential of the first kind on $S_{\Gamma}$:
\begin{align}
\omega_{g}=\sum_{h\in C(|g)}d\log \langle h(z^+(g)), h(z^-(g)), z,
z_0 \rangle.
\end{align}
This differential does not depend on $z_0$, and depends on $g$
additively. Also If the class of $g$ is primitive ( i.e. non -
divisible in H:$=\Gamma/[\Gamma, \Gamma]$ ), $\omega_g$ can be
rewritten as following
\begin{align}
\omega_{g}=\sum_{h\in S(g)}d\log \langle z^+(h), z^-(h), z, z_0
\rangle.
\end{align}
(b) If $g_i$ form a part of the marking of $\Gamma$, and $a_i$ are
the homology classes described before, we have
\begin{align}
\int_{a_i}\omega_{g_j}=2\pi i\delta_{ij}.
\end{align}
It follows that the map $g$ mod $[\Gamma, \Gamma]\mapsto \omega_g$
embeds $H$ as a sublattice in the space of all differentials of the
first kind.\\

(c) Denote by $\{b_j\}$ the complementary set of homology classes in
$H_1(S_{\Gamma}, \mathbb{Z})$ as in before. Then we have for $i\neq
j$, with an appropriate choice of logarithm branches:
\begin{align}
\tau_{ij}:=\int_{b_i}\omega_{g_j}=\sum_{h\in C(g_i|g_j)}\log \langle
z^+(g_i),z^-(g_i), h(z^+(g_j)), h(z^-(g_j))\rangle\label{F1.10}
\end{align}
 And
\begin{align}
\tau_{ii}=\log q(g_i)+\sum_{h\in C_0(g_i|g_i)}\log \langle z^+(g_i),
z^-(g_i),h(z^+(g_i)), h(z^-(g_i))\rangle\label{F1.11}
\end{align}
where $C_0(g_i|g_i)$ is  $C(g_i|g_i)$ without the identity class.
\end{prop}
\begin{proof} For the proofs, see \cite{Man1} and \cite{Man2}. Notice that
our notation here slightly differs from \cite{Man1}; in particular,
$\tau_{ij}$ here corresponds to $2\pi i\tau_{ij}$ of
\cite{Man1}.\end{proof}
\subsubsection{Differentials of the third kind and Green's functions} Lets consider the points $a$ and
$b$ in $X_\Gamma$ the fundamental domain of $\Gamma$ and put
$\nu_{(a)-(b)}=\sum_{h\in\Gamma} d\log \langle a, b, h(z), h(z_0)
\rangle$. Then assuming $a(\Gamma)<1$, we see that this series
absolutely converges and because it's $\Gamma$ - automorphic it
gives us a differential of the third kind on $S_{\Gamma}$ with
residues $\pm 1$ at the images of $a,b$ in $S_{\Gamma}$. Moreover,
since both points $a,b$ are out of the circles $C_i$,  its $a_i$ -
periods vanish. Now, if we consider the linear combination
$\nu_{(a)-(b)}-\sum_{j=1}^p X_j(a, b)\omega_{g_j}$
 with real coefficients $X_j$, then it will have pure
imaginary $a_i$ - periods and  If we find  real coefficients $X_j$
so that the real part of $b_i$ - periods of the form
\begin{align}
\omega_{(a)-(b)}=\nu_{(a)-(b)}-\sum_{j=1}^pX_j(a,b)\omega_{g_j}
\end{align}
 vanish, we will be able to use this differential in order
to calculate conformally invariant Green's functions. The set of the
equations for calculating the coefficients $X_j(a, b)$ are as
following
\begin{align}
\sum_{j=1}^pX_j(a,b)Re \tau_{ij}=Re\int_{b_i}\nu_{(a)-(b)}
=\sum_{h\in S(g_i)}\log |\langle a, b, z^+(h), z^-(h)\rangle|
\end{align}
for $i=1,...,p$. The parts  $Re \tau_{ij}$ are calculated by means
of formula (\ref{F1.10}) and (\ref{F1.11}), and $b_i$ - periods
of $\nu_{(a)-(b)}$  are given in \cite{Man1}.\\

Now if we denote the points $a,b,c$, and $d$ in $S_{\Gamma}$ and
their images in the fundamental domain $X_{\Gamma}$ by the same
notations, then we have
\begin{align}
Re\int_d^c\nu_{(a)-(b)}= \sum_{h\in\Gamma}\log\mid\langle a, b,
h(c),h(d)\rangle\mid,~Re\int_d^c\omega_{g_j}= \sum_{h\in
S(g_j)}\log\mid\langle
 z^+(h), z^-(h), c, d \rangle\mid\nonumber
\end{align}
And finally we have the Green's function on $S_\Gamma$ as following
\begin{align}
g((a)-(b),(c)-(d))&=Re\int_d^c\omega_{(a)-(b)}\nonumber \\
&=\sum_{h\in\Gamma}\log\mid\langle a, b, h(c),
h(d)\rangle\mid\nonumber\\
&-\sum_{j=1}^pX_j(a,b)\sum_{h\in S(g_j)}\log\mid\langle
 z^+(h), z^-(h), c, d \rangle\mid.
\end{align}
\section{Two boundary component case}
All hyperbolic Riemann surfaces are uniformized by Fuchsian groups.
Also Fuchsian groups act on hyperbolic space by Poincare extension
and uniformize some hyperbolic 3-manifolds with two boundaries at
infinity with the same genera. But they do not uniformize all such
manifolds; In fact they are uniformized by an extension of Fuchsian
groups that are called Quasi-Fuchsian groups. In this chapter first
we intend to extend the method in previous chapter to fuchsian
groups and then
for Quasi-fuchsian groups.\\

Consider hyperbolic plane $H^2$ with the upper half plane model and
coordinate $z=x+iy\in \mathbb{C}$ with $y>0$, equipped by the
hyperbolic distance function corresponding to the metric
\begin{align}
ds^2=\frac{|dz|^2}{y^2}
\end{align}
 of the constant curvature -1. The geodesics in this model
are vertical half-lines $x=constant$ and vertical half-circles
orthogonal to the line at infinity $\mathbb{R}$ ( i.e, for $y=0$ ).\\

If we consider the end points of geodesics ( including $\infty$ for
the end of vertical half-lines other than the end in $\mathbb{R}$ )
then we can consider circle $\hat{\mathbb{R}}=\mathbb{R}\cup
\{\infty\}$ ( or $\mathbb{P}^1(\mathbb{R})=S^1$  in unit disc model
) as boundary at infinity of $H^2$.
\subsection{Fuchsian Groups} Each subgroup $F$ of
$PGL(2,\mathbb{R})$, the general projective linear group over
$\mathbb{R}$, that acts freely and properly discontinuously on $H^2$
is called a Fuchsian group. Similar to the Kleinian groups the limit
set $\Lambda(F)$ and the region of discontinuity $\Omega(F)$ are
defined but in $\hat{\mathbb{R}}$, the boundary at infinity of
$H^2$. And the properties are similar to them. Also similar to the
loxodromic elements, a hyperbolic element is an element $h\in
PGL(2,\mathbb{R})$ with two fixed points $z^{\pm}(h)$ in
$\hat{\mathbb{R}}$. For a Fuchsian group $F$ the quotient space
$H^2/F$ is a hyperbolic Riemann surface  possibly with boundary $\Omega(F)/F$ ( If $\Omega(F)\neq\emptyset$ ).\\

Let $F$ be a Fuchsian group such that $S=H^2/F$ is a compact Riemann
surface with the genus $p>1$. Then $F$ is a purely hyperbolic group
of finite order $2p$ (for example see \cite{CS}). Lets denote by
$f_i$ the generators of $F$ and by $Fix(f_i)=\{z^{\pm}(f_i)\}$ the
fixed point set of $f_i$. Also consider that $P$ is the fundamental
polygon of $F$ with $a_1,b_1,a'_1,b'_1,\ldots,b'_{p}$ as it's sides,
such that $f_i(a_i)=a'_i$ and $f_{p+i}(b_i)=b'_i$ for
$i=1,\ldots,p$. We represent $P$ as following:
\begin{align}
P=a_1(x_1)b_1(y_1)a'_1(x'_1)b'_1(y'_1)\ldots a'_p(x'_p)b'_p(y'_p)
\end{align}
Where $x_i,  x'_i  ,y_i$  and $y'_i$ are the intersection points of
$a_i$ and $b_i$, $a_i'$ and $b_i'$ and so on. Now we have
$$\pi_1(S)\simeq F = \langle \{f_i\ ; i=1,\ldots,2p \};\prod_1^p
[f_i,f_{p+i}]=I \rangle$$ And also ( See for example \cite{CS}).
$$H_1(S,\mathbb{Z}) =
\frac{\pi_1(S)}{\langle\left\{aba^{-1}b^{-1}|a,b\in\pi_1(S)\right\}\rangle}
  = \mathbb{Z}\oplus\ldots\oplus\mathbb{Z}\qquad( 2p~~ times )$$
If we denote by ${\overline a}_i$ and ${\overline b}_i$ the images
of $a_i$ and $b_i$ in $S$, then $\{{\overline a}_i, {\overline
b}_j\}$ generate $H_1(S,\mathbb{Z})$.
\subsection{Extension of Fuchsian group on $H^3$} $F$ acts similarly, freely and
properly discontinuously on lower half plane $-H^2$ of $\mathbb{C}$
too, and $-H^2$ is $F$-invariant. Lets denote by $F$ the Poincare
extension of $F$ on $H^3$ too (see \cite{HK} or \cite{FH}). By this
extension $F$ can be considered as a Kleinian group. And we have
$$N_F = \frac{H^3}{F}\cong\frac{H^2}{F}\times(0,1) =
  S\times(0,1)\qquad and\qquad F = \pi_1(N_F) = \pi_1(S)$$
 (See \cite{HK}). And $N_F$ has a hyperbolic structure,
\cite{Mc}.
\subsection{F invariants} We know that $\Lambda(F)$ the limit point
set of $F$ is $\mathbb{\hat{R}}$ ( $S^1$ in unit disc model ),
\cite{B}. Then as a Kleinian group
$\Lambda(F)=\mathbb{\hat{R}}\subset\hat{\mathbb{C}}$ and the region
of discontinuity of $F$ considering as a Kleinian group is
\begin{align}
\Omega(F)=\mathbb{\hat{C}}\backslash\Lambda(F)=\mathbb{\hat{C}}\backslash
\mathbb{\hat{R}}=-H^2\cup H^2
\end{align}
  And also the Kleinian manifold is
\begin{align} M_F =\frac{H^3\cup(-H^2\cup H^2)}{F}= S\times[0,1].
\end{align}
Now because $S$ is compact then $\partial S = \emptyset$.
Consequently $N_F$ is a hyperbolic 3 - manifold with two compact
boundary component at infinity $S_0=H^2/F=S\times\{0\}$ and
$S_1=-H^2/F=S\times\{1\}$ with the same genus $p$.
\subsection{Coding of the points of $\Lambda(F)$ and the Geodesics } For
convenience and having a simple intuition
lets consider the unit disc model for $H^2$ in this section. We can
code points of $\Lambda(F)=S^1$ and geodesics with beginning and end
points on $\Lambda(F)=S^1$ as following:\\

Lets mark semicircles including sides of P by $c_1,...,c_{4p}$ in
the counter clockwise direction around $S^1$ and put $g_1=f_1,
g_2=f_{p+1}, g_3=f^{-1}_1, g_4=f^{-1}_{p+1}, g_5=f_2, g_6=f_{p+2},
g_7=f_2^{-1}, g_8=f_{p+2}^{-1}$, and so on. In general for
$k=0,...,p-1$, $g_{4k+1}=f_{k+1}, g_{4k+2}=f_{p+k+1},
g_{4k+3}=f^{-1}_{k+1}, g_{4k+4}=f^{-1}_{p+k+1}$. And label end
points of $c_i$ on $S^1$ by $P_i$ and $Q_{i+1}$   ( with
$Q_{4p+1}=Q_1$ ) with $P_i$ occurring  before $Q_{i+1}$ in the
counter clockwise direction, ( See figure \ref{Fig2.13} ). And
define
\begin{figure}[!hbtp]
\centerline{\includegraphics{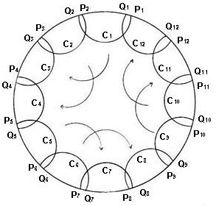}} \caption { }\label{Fig2.13}
\end{figure}
$$f_F : S^1 \rightarrow S^1$$
$$f_F (x) = g_i(x)\hspace{5mm} if\hspace{5mm} x\in [P_i,P_{i+1})$$
Then $f_F$ is a well defined map and is called a Markov map related
to the Fuchsian group $F$. We have the following lemma
\begin{lem} The map $f_F$ and the group $F$ are orbit equivalent
on $S^1$, namely except for the pairs $(Q_i, g_{i-1}Q_i)$, for
$i=1,2,...,4p$; for each x,y in $S^1$, $x=f(y)$ for some $f$ in $F$
if and only if there exists nonnegative integers $n,m$ such that
$f_F ^n(x)=f_F^m(y)$.
\end{lem}
\begin{proof} ( See \cite{BS}).\end{proof}
Now label each arc $[P_i, P_{i+1})$ by $g_i$, and for each element
$x$ in $S^1$ put $x=(...,x_2,x_1,x_0 )$. Where the component $x_n$
is the label of the segment to which $f_F ^n(x)$ belongs. We know
that for each generator $f_i$ of $F$, the fixed point $Z^+(f_i)$ is
in the circle $a_i'\cup(-a_i')$ and $Z^-(f_i)$ is in $a_i\cup(-a_i)$
then from the definition of Markov map we have
\begin{align}
Z^+(f_i)=(...,f_i^{-1},f_i^{-1},f_i^{-1} )\qquad and\qquad
Z^-(f_i)=(...,f_i,f_i,f_i )\label{F2.10}
\end{align}
 Next, if $\{x, y\}$ be a geodesic with the beginning
point $x\in S^1$ and end point $y=(...,y_2,y_1,y_0 )\in S^1$ then
put
\begin{align}
 \{x, y\} = (...,x_2,x_1,x_0,y_0^{-1},y_1^{-1},y_2^{-1},...).
\end{align}
Now for each $h\in F$  let $\gamma_h=\{z^+(h),z^-(h)\}$ show the
geodesic arc between $z^+(h)$ and $z^-(h)$ (the axis of $h$) and
${\overline\gamma}_h$ the image of $\gamma_h$ in $N_F$. Then we have
the following geodesics in $N_F$
\begin{itemize}
    \item{Closed geodesics: a geodesic in $N_F$ is closed if and only if it's the projection of the axis
    of a hyperbolic element in $F$},
    \item{Images of the geodesics with beginning and end points in $\Lambda(F)$},
    \item{Geodesics with beginning points in $\partial N_F$ that are limit cycle to
    ${\overline\gamma}_{f_i}$,} for a generator $f_i$ of $F$,
    \item{Geodesics with beginning and end points in $S_0$ or $S_1$.},
    \item{Geodesics with beginning point in $S_0$ and end point in $S_1$ and vis
    versa.}
\end{itemize}
\subsection{Schottky groups associated to the Fuchsian groups} For
the Riemann surface $S$ and associated Fuchsian group $F$ as above
let $N$ be the smallest normal subgroup of $F$ including $f_i$ for
$i=p+1,...,2p$. Then it's obvious that the factor group $F/N$ is a
free group generated by $p$ generators. If we consider the covering
$\tilde{S}\rightarrow S$ associated to $N$, then from normality of
$N$ the group of deck transformations of this cover is $F/N$ and
according to the classical Koebe uniformization theorem \cite{K}(
see also \cite{F})there is a planar region
$\Omega\subset\mathbb{\hat{C}}$ that is a region of discontinuity of
a Schottky group $\Gamma = <\gamma_1,\gamma_2,..., \gamma_p>$ and
$F/N\simeq \Gamma$. $\tilde{S}$ is covering isomorphic to
$\Omega=\Omega(\Gamma)$ and for coverings $J:H^2\longrightarrow
\Omega(\Gamma)$, $\pi_\Gamma:\Omega(\Gamma)\longrightarrow
\Omega(\Gamma)/\Gamma$ and $\pi_F:H^2\longrightarrow H^2/F$ we have
$\pi_\Gamma \circ J=\pi_F$ i.e. the following diagram is commutative
\begin{center}
$H^2\hspace{5mm}\overset{J}{\longrightarrow}\hspace{5mm}\Omega(\Gamma)$\\
$\pi_F\searrow\hspace{7mm}\swarrow \pi_\Gamma$\\
$S$
\end{center}
 Also we have  $J\circ f_i=\gamma_i\circ J$ for
$i=1,...,p$ and $J\circ f=J$ for each $f$ in $N$ and each mappings
is complex-analytic covering and $\Gamma$ is uniquely determined to
within conjugation in $PGL(2,\mathbb{C})$ \cite{Z}.\\

For an extension of $J$ to a set some more than  $H^2$ that we need
it in the later computations, lets denote by $F_0$ the free subgroup
of $F$ generated by $f_1,...,f_p$. We know that $\Lambda(F_0)$ is a
subset of $\Lambda(F)$. $\Lambda(F_0)$ is $F_0$ invariant and is out
of the closer of the fundamental domain of $F_0$, then all points of
$\Lambda(F_0)$ are in only the circles $C_i$ that are related to the
generators of $F_0$ and their inverses. This shows that the coding
of each point in $\Lambda(F_0)$ includes only the generators of
$F_0$ and their inverses. Also since each generator of $F$ like $f$
is an isometry of $H^2$, it's a composition of two reflections, then
$f(x)$ is not in the isometry circle of $f^{-1}$ for each point $x$
in $S^1$. Then by the definition of the Markov map $f_F$ all the
codings of the points of $S^1$ are irreducible. This means that each
point $x$ in $\Lambda(F_0)$ can be coded by the irreducible formal
combination
$x=(...,f_{i_2}^{\epsilon_{i_2}},f_{i_1}^{\epsilon_{i_1}},f_{i_0}^{\epsilon_{i_0}})$
for $i_j\leq p$. This Fuchsian coding is suitable for expressing
geodesics in $N_F$ or $S_F$ when the beginning and end points of the
geodesics are in $\Lambda(F_0)$ and also can be used for extending
the map $J$ on $\Lambda(F_0)$ and maybe on some more.\\

Now we want to extend $J$ to $H^2\cup\Lambda(F_0)$ ( onto
$\overline{\Omega(\Gamma)}=\hat{\mathbb{C}}$ ). For this, we know
that each element of $\Lambda(\Gamma)$ can be coded by the infinite
words
$z=...\gamma_{i_2}^{\epsilon_{i_2}}\gamma_{i_1}^{\epsilon_{i_1}}\gamma_{i_0}^{\epsilon_{i_0}}(z_0)$
for a constant $z_0$ in $\Omega(\Gamma)$ and $\epsilon_{i}=\pm 1$.
Similarly,  since $F_0$ is free and purely loxodromic group then
it's a Schottky group too and we can use Schottky coding for it too.
For conveniences in some proofs we will use Schottky codings for the
points of $\Lambda(F_0)$. Then each point $x$ in $\Lambda(F_0)$ can
be coded by the infinite words
$x=...f_{i_2}^{\epsilon_{i_2}}f_{i_1}^{\epsilon_{i_1}}f_{i_0}^{\epsilon_{i_0}}(x_0)$
for $i_j\leq p$ and a constant point $x_0$ in
$\Omega(F_0)\subset\hat{\mathbb{C}}$ ( also $H^2\subset\Omega(F_0)$
). Then we can consider $x_0$ in $\Omega(F_0)\cap H^2$ such that
$z_0=J(x_0)$. Also from (\ref{F2.10}) and using the properties of a
loxodromic element and invariance of the limit set under $F$ ( or
$F_0$ ) it is not hard to show that on $\Lambda(F_0)$ we have the
following relation between Fuchsian and Schottky coding
\begin{align}
...f_{i_k}^{\epsilon_{i_k}}...f_{i_1}^{\epsilon_{i_1}}f_{i_0}^{\epsilon_{i_0}}f_{i_k}^{\epsilon_{i_k}}
...&f_{i_1}^{\epsilon_{i_1}}f_{i_0}^{\epsilon_{i_0}}(x_0)=Z^+(f_{i_k}^{\epsilon_{i_k}}...f_{i_1}^{\epsilon_{i_1}}f_{i_0}^{\epsilon_{i_0}})\nonumber\\
&=(...,f_{i_0}^{-\epsilon_{i_0}},...,f_{i_{k-1}}^{-\epsilon_{i_{k-1}}},f_{i_k}^{-\epsilon_{i_k}}
,f_{i_0}^{-\epsilon_{i_0}},...,f_{i_{k-1}}^{-\epsilon_{i_{k-1}}},f_{i_k}^{-\epsilon_{i_k}})\label{F2.9}
\end{align}
And also this is true for repelling points too because of the
relation $Z^-(f)=Z^+(f^{-1})$ for each $f$ in $F$ . Now, lets put
\begin{align}
J(...f_{i_2}^{\epsilon_{i_2}}f_{i_1}^{\epsilon_{i_1}}f_{i_0}^{\epsilon_{i_0}}(x_0))
=...\gamma_{i_2}^{\epsilon_{i_2}}\gamma_{i_1}^{\epsilon_{i_1}}\gamma_{i_0}^{\epsilon_{i_0}}(z_0)
\end{align}
in other words, by the definition of infinite words, the continuity
of $J$ and the relation $J\circ f(x)=\gamma\circ J(x)$  on $H^2$
\begin{align}
J(...f_{i_2}^{\epsilon_{i_2}}f_{i_1}^{\epsilon_{i_1}}f_{i_0}^{\epsilon_{i_0}}(x_0))
&=J(\lim_{j\rightarrow\infty}f_{i_j}^{\epsilon_{i_j}}...f_{i_2}^{\epsilon_{i_2}}f_{i_1}^{\epsilon_{i_1}}f_{i_0}^{\epsilon_{i_0}}(x_0))\nonumber\\
&=\lim_{j\rightarrow\infty}Jf_{i_j}^{\epsilon_{i_j}}...f_{i_2}^{\epsilon_{i_2}}f_{i_1}^{\epsilon_{i_1}}f_{i_0}^{\epsilon_{i_0}}(x_0)\nonumber\\
&=\lim_{j\rightarrow\infty}\gamma_{i_j}^{\epsilon_{i_j}}...\gamma_{i_2}^{\epsilon_{i_2}}\gamma_{i_1}^{\epsilon_{i_1}}
\gamma_{i_0}^{\epsilon_{i_0}}(J(x_0))\nonumber\\
&=...\gamma_{i_2}^{\epsilon_{i_2}}\gamma_{i_1}^{\epsilon_{i_1}}\gamma_{i_0}^{\epsilon_{i_0}}(z_0)\nonumber\\
\end{align}
Then the map $J:\Lambda(F_0)\rightarrow \Lambda(\Gamma)$ is well
defined and  onto. Also from (\ref{F2.9}) we see that we can explain
$J$ and consequently the functions ( spatially Green's function ) on
$S_0$  by the Fuchsian coding, when these functions  are expressed
via this map.\\

For each point $x$ in $\Lambda(F_0)$ and element $f$ in $F_0$
equivalent to $\gamma$ (i.e. $
f=f_{i_k}^{\epsilon_{i_k}}...f_{i_1}^{\epsilon_{i_1}}f_{i_0}^{\epsilon_{i_0}}\sim
\gamma=\gamma_{i_k}^{\epsilon_{i_k}}...\gamma_{i_1}^{\epsilon_{i_1}}\gamma_{i_0}^{\epsilon_{i_0}}$,
where $i_j\leq p$ and  $f_i$ is replaced by $\gamma_i$ and vice
versa ) we have
\begin{align}
J\circ
f(...f_{i_2}^{\epsilon_{i_2}}f_{i_1}^{\epsilon_{i_1}}f_{i_0}^{\epsilon_{i_0}}(x_0))
=\gamma\circ
J(...f_{i_2}^{\epsilon_{i_2}}f_{i_1}^{\epsilon_{i_1}}f_{i_0}^{\epsilon_{i_0}}(x_0))\nonumber
\end{align}
Then $J\circ f=\gamma\circ J$ and $J(z^\pm(f))=z^\pm(\gamma)$.
\subsection{Automorphic  Functions on $S_0$} Let $F_0$ be the free group
generated by $f_1,...,f_p$ and $D=\Sigma m_x(x)$ be a divisor with
support $|D|$ in $H^2$. Put $D_0:=\Sigma m_{x}(Jx)$ and  again lets
denote by $w_A$ a meromorphic function on $\mathbb{P}^1(\mathbb{C})$
with the divisor $A$, and for an element $x_0$ in $H^2\backslash
\bigcup_{f\in F_0}f(|D|)$ define:
\begin{align}
W_{D,x_0}(x)=\prod_{f\in
F_0}\frac{w_{D_0}(Jf(x))}{w_{D_0}(Jf(x_0))}.\label{F2.91}
\end{align}
\begin{thm} For the Schottky group  $\Gamma$ associated to the Fuchsian group $F$, if $a(\Gamma)<1$, then
the product (\ref{F2.91}) converges absolutely
and uniformly on any compact subset of $H^2$, after deleting a
finite number of factors that may have a pole or zero on  this
subset.
\end{thm}
\begin{proof} Let $K$ be a compact subset of $H^2$. Since $F$ acts on $H^2$ properly discontinuously
the set $K_{F_0}=\{f\in F_0|f(K)\cap |D|\neq\emptyset\}$ is finite.
If we delete the set $K_{F_0}$ from the index set of the product
(\ref{F2.91}) then for each $f\in F_0\setminus K_{F_0}$ when $f(x)$
and $f(x_0)$ lie outside a fixed compact neighborhood of $|D|$ we
have
\begin{align}
\biggm|\frac{w_{D_0}(Jf(x))}{w_{D_0}(Jf(x_0))}-1\biggm|&\leq
c|Jf(x)-Jf(x_0)|=c|\alpha(z)-\alpha(z_0)|\nonumber\\
&\leq\frac{c}{2}|z-z_0|\biggm|\frac{d\alpha(z)}{dz}+\frac{d\alpha(z_0)}{dz}\biggm|\nonumber
\end{align}
Where $c$ is a constant, $z=J(x)$,  $z_0=J(x_0)$ and
$f\sim\alpha=\left(\begin{array}{cc}a&b\\c&d\end{array}\right)$ (
and $ad-bc=1$). The last inequality comes from the equality
$\frac{1}{|cz+d|^2}=|\frac{d\alpha(z)}{dz}|$. Now since
$a(\Gamma)<1$ the series
\begin{align}
\sum_{\gamma\in\Gamma}\biggm|\frac{d\gamma(z)}{dz}\biggm|
\end{align}
converges uniformly on the compact subsets of $\Omega(\Gamma)$. Then
the product (\ref{F2.91}) is convergent.
\end{proof}
In $W_{D,x_0}(x)$, if we change the point $x_0$ to $x_1$, then  we
have $W_{D,x_0}(x)=C^{x_1}_{x_0}W_{D,x_1}(x)$. Where $C^{x_1}_{x_0}$
is a nonzero complex number that depends on the points $x_0$ and
$x_1$ and $C^{x_1}_{x_0}C^{x_0}_{x_1}=1$. Also, for $f\in F$ and
$x\in H^2$ we have
$$W_{D,x_0}(fx)=\prod_{h\in F_0}\frac{w_{D_0}(Jh(x_0)))}{w_{D_0}(Jhf^{-1}
(x_0))}\prod_{h\in F_0}
\frac{w_{D_0}(Jh(x))}{w_{D_0}(Jh(x_0))}=\mu_D(f)W_{D,x_0}(x).$$
$\mu_D(f)$ is a nonzero complex number multiplicative on $D$ and $f$
and also independent of $x_0$. We can see this as following
$$\mu_D^{x_0}(f)W_{D,x_0}(x)=W_{D,x_0}(fx)=C^{x_1}_{x_0}\mu_D^{x_1}(f)C^{x_0}_{x_1}W_{D,x_0}(x)=\mu_D^{x_1}(f)W_{D,x_0}(x)$$
Then $\mu_D^{x_0}(f)=\mu_D^{x_1}(f)$.  When $\Gamma$ is a cyclic
group $\mu_D(f)=1$. This shows that $W_{D,x_0}$ is not $F$
automorphic function on $H^2$ in general.
\begin{thm} a) Lets denote by $C(f|g)$ a set of the representatives of
$(f^n)\backslash F_0/(g^n)$. Then for $f\neq g~ mod[F_0,F_0]$ in
$F_0$ we have
\begin{align}
\mu_{(g(x_1))-(x_1)}(f)&=\prod_{h\in
F_0}\frac{w_{(Jg(x_1))-(Jx_1)}(J(h(x_0)))}{w_{(Jg(x_1))-(Jx_1)}(J(h\circ
f^{-1}(x_0)))}\nonumber\\
&=\prod_{h\in
C(f|g)}\frac{w_{(J(z^+(f)))-(J(z^-(f)))}(Jh(z^+(g)))}{w_{(J(z^+(f)))-(J(z^-(f)))}(Jh(z^-(g)))}\label{F2.11}
\end{align}
And by defining  $Q(f):=\langle J(z^+(f)), J(z^-(f)), J(f(x_1)),
J(x_1)\rangle$ we have
\begin{align}
\mu_{(f(x_1))-(x_1)}(f)=Q(f)\prod_{h\in
C_0(f|f)}\frac{w_{(J(z^+(f)))-(J(z^-(f)))}(Jh(z^+(f)))}{w_{(J(z^+(f)))-(J(z^-(f)))}(Jh(z^-(f)))}\label{F2.12}
\end{align}
Where $C_0(f|f)$ is the set $C(f|f)$ without the identity class.\\
b) Lets denote by $C(|f)$ a set of representatives of $F_0/(f^n)$
and by $S(f)$ the conjugacy class of $f$ in $F_0$. Then for some
$x_1$ in $H^2$ such that $z_1=J(x_1)$ stays in
$\Omega(\Gamma)\setminus \Gamma_\infty$, we have
\begin{align}
W_{(f(x_1))-(x_1),x_0}(x)&= \prod_{h\in
C(|f)}\frac{w_{(Jh(z^+(f)))-(Jh(z^-(f)))} (J(x))}
{w_{(Jh(z^+(f)))-(Jh(z^-(f)))} (J(x_0))}\nonumber\\
&=\prod_{h\in S(f)}\frac{w_{(J(z^+(h)))-(J(z^-(h)))}
(J(x))}{w_{(J(z^+(h)))-(J(z^-(h)))} (J(x_0))}.\label{T2.11}
\end{align}
And this is independent of $x_1$.\end{thm}
\begin{proof} Lets put $J(g(x_1))=a$, $J(x_1)=b$ and
$w_{(a)-(b)}(x)=\frac{a-x}{b-x}$ then we have
\begin{align}
\mu_{(g(x_1))-(x_1)}(f)&=\prod_{h\in
F_0}\frac{a-Jh(x_0)}{b-Jh(x_0)}\big/\frac{a-Jhf^{-1}(x_0)}{b-Jhf^{-1}(x_0)}\nonumber\\
&=\prod_{h\in
C(f|g)}\prod_{m=-\infty}^{m=\infty}\prod_{n=-\infty}^{n=\infty}\frac{a-J(f^{-m}hg^n(x_0))}{a-J(f^{-m}hg^{n-1}(x_0))}
\frac{b-J(f^{-m}hg^{n-1}(x_0))}{b-J(f^{-m}hg^{n}(x_0))}\nonumber\\
&=\prod_{h\in
C(f|g)}\prod_{m=-\infty}^{m=\infty}\prod_{n=-\infty}^{n=\infty}\frac{A_n}{A_{n-1}}
\frac{B_{n-1}}{B_n}\nonumber\\
&=\prod_{h\in
C(f|g)}\prod_{m=-\infty}^{m=\infty}\frac{A_\infty}{A_{-\infty}}
\frac{B_{-\infty}}{B_\infty}\qquad\left(\prod_{-N}^{N}\frac{A_n}{A_{n-1}}\frac{B_{n-1}}{B_n}=\frac{A_N}{A_{-N-1}}\frac{B_{-N-1}}{B_N}\right)\nonumber\\
&=\prod_{h\in
C(f|g)}\prod_{m=-\infty}^{m=\infty}\frac{a-J(f^{-m}h(z^+(g))}{a-J(f^{-m}h(z^-(g)))}
\frac{b-J(f^{-m}h(z^-(g)))}{b-J(f^{-m}h(z^+(g)))}\nonumber\\
&=\prod_{h\in
C(f|g)}\prod_{m=-\infty}^{m=\infty}\frac{J(f^{m+1}(x_1))-J(h(z^+(g))}{J(f^{m+1}(x_1))-J(h(z^-(g)))}
\frac{J(f^{m}(x_1))-J(h(z^-(g)))}{J(f^{m}(x_1))-J(h(z^+(g)))}\label{F2.8}\\
&=\prod_{h\in
C(f|g)}\frac{J(z^+(f))-J(h(z^+(g))}{J(z^-(f))-J(h(z^+(g)))}
\frac{J(z^-(f))-J(h(z^-(g)))}{J(z^+(f))-J(h(z^-(g)))}\nonumber\\
&=\prod_{h\in
C(f|g)}\frac{w_{(J(z^+(f)))-(J(z^-(f)))}(Jh(z^+(g)))}{w_{(J(z^+(f)))-(J(z^-(f)))}(Jh(z^-(g)))}\nonumber
\end{align}
The part (\ref{F2.8}) comes from the equation $J\circ g=\gamma\circ
J$ and the invariance of the Cross-Ratio on the action of Mobius
transformations. The second part of a) and b) can be proved
similarly. For b) we should consider that
$h(z^{\pm}(g))=z^{\pm}(hgh^{-1})$.
\end{proof}
In above theorem in the case that the divisor is $(a)-(b)$ we have
\begin{align}
\mu_{(a)-(b)}(f)=\prod_{h\in
S(f)}\frac{w_{(J(a))-(J(b))}(J(z^+(h)))}{w_{(J(a))-(J(b))}(J(z^-(h)))}\label{F2.13}.
\end{align}
By previous theorem $W_{(f(x_1))-(x_1),x_0}(x)$ is a meromorphic
function without any poles and zeroes in $H^2$ then it is a
holomorphic function on $H^2$. Also as we see in the expression, it
is independent of the point $x_1$.
\subsection{Differentials of the first kind on $S_0$} For each $f$ in $F_0$ lets put
$$\omega_f = d\log W_{(f(x_1))-(x_1),x_0}(x).$$
However $W_{D,x_0}$ is not $F$- automorphic function on $H^2$ in
general but we have $$d\log W_{D,x_0}(fx)=d\log W_{D,x_0}(x).$$ And
this shows that the differential $d\log W_{D,x_0}(x)$ is $F$
automorphic. Then $\omega_f$ is a differential of the first kind on $S_0$.\\

According to the classical theorem of cuts we can choose
$\Omega(\Gamma)$ the region of discontinuity of the Schottky group
$\Gamma$ with the marking $\{D_1,...,D_{2p}; \gamma_1,\gamma_2,...,
\gamma_p \}$ with $C_i=\partial D_i$ such that $\bar{a}_i$ the image
of $a_i$ for $i=1,...,p$ in $S_0$ be coincident with the image of
$C_{p+i}$. And these together with $\bar{b}_i$ the image of $b_i$
for $i=1,...,p$ in $S_0$ make a canonical base for $H_1(S_0,\mathbb
Z)$ i.e.
\begin{align}
(\bar{a}_i,\bar{a}_j)=(\bar{b}_i,\bar{b}_j)=0\hspace{1cm}and
 \hspace{1cm}(\bar{a}_i,\bar{b}_j)=\delta_{ij}.
\end{align}
\begin{prop} a)  $\{\omega_{f_i}\}_1^p$ is a Riemann's basis for
the space of differentials of the first kind on $S_0$ by choosing
the previous base for $H_1(S_0,\mathbb Z)$ i.e.
\begin{align}
\frac{1}{2\pi i}\int_{\bar{a}_j}\omega_{f_i}=\delta_{ij}.
\end{align}
 b) If we denote by $C(f|g)$ a set of the representatives
of $(f^n)\backslash F_0/(g^n)$ and by $C_0(f|g)$, the set $C(f|g)$
without the identity class, Then for $i\neq j$ we have
\begin{align}
\tau_{ij}:=\int_{\bar{b}_j}\omega_{f_i}=\sum_{h\in C(f_j|f_i)}\log
\langle J(z^+(f_j)),J(z^-(f_j)), J(h(z^+(f_i)),J(h(z^-(f_i))\rangle
\end{align}
and for $i=j$ by defining $Q(f_i):=\langle J(z^+(f_i)), J(z^-(f_i)),
J(f_i(x_1)), J(x_1)\rangle$ we have
\begin{align}
\tau_{ii}=\log Q(f_i)+\sum_{h\in C_0(f_i|f_i)}\log \langle
J(z^+(f_i)),J(z^-(f_i)), J(h(z^+(f_i)),J(h(z^-(f_i))\rangle.
\end{align}
\end{prop}
\begin{proof} We have
\begin{align}
W_{(f(x_1))-(x_1),x_0}(x)=\prod_{\alpha\in \Gamma}\frac{w_{(\gamma(
z_1))-(z_1)}(\alpha(z))}{w_{(\gamma(z_1))-(z_1)}(\alpha(z_0))}
\end{align}
where $z=J(x), z_0=J(x_0)$,  $z_1=J(x_1)$ and $f\sim\gamma$. If we
show this equality for $f_i,~ i=1,2,3,...,p$ by
\begin{align}
W_{(f_i(x_1))-(x_1),x_0}(x)=\overline{W}_{(\gamma_i(z_1))-(z_1),z_o}(z)
\end{align}
then for a) we have
\begin{align}
\frac{1}{2\pi i}\int_{\bar{a}_j}\omega_{f_i}&=\frac{1}{2\pi i}\int_{\bar{a}_j}d\log W_{(f_i(x_1))-(x_1),x_0}(x)\nonumber\\
&=\frac{1}{2\pi i}\int_{a_j}\frac{dJ}{dx}~\frac{d\log W_{(f_i(x_1))-(x_1),x_0}(x)}{dz}dx\nonumber\\
&=\frac{1}{2\pi i}\int_{J(a_j)}\frac{d\log\overline{W}_{(\gamma_i(z_1))-(z_1),z_o}(z)}{dz}dz\nonumber\\
&=\frac{1}{2\pi i}\int_{C_{p+j}}d\log\overline{W}_{(\gamma_i(z_1))-(z_1),z_o}(z)\nonumber\\
&=\frac{1}{2\pi i}\sum_{\alpha\in C(|\gamma_i)}\int_{C_{p+j}}d\log
w_{(\alpha(z^+(\gamma_i)))
-(\alpha(z^-(\gamma_i)))}(z)\nonumber\\
&=\sum_{\alpha\in
C(|\gamma_i)}\begin{cases}1&if~\alpha(z^+(\gamma_i))\in
D_{p+j},\alpha(z^-(\gamma_i))\notin
D_{p+j}\\-1&if~\alpha(z^+(\gamma_i))\notin
D_{p+j},\alpha(z^-(\gamma_i))\in
D_{p+j}\\0&otherwise\end{cases}\nonumber
\end{align}
The  last equality comes from this fact that  if $i=j$ since
$z^+(\gamma_i))\in D_{p+i}$ and $z^-(\gamma_i)\in D_{i}$, only for
$\alpha=id$ the first alternative and for all $\alpha\neq id$ the
third one is valid. If $i\neq j$ only the third alternative is right
for all $\alpha$. We can see these from the Figure 1.\\

For the first part of b). We have shown by the point $x_j$ the
intersection of $a_j$ and $b_j$ in the representation of $P$, the
fundamental domain of $F$. Then $\bar b_j$ is the image of the part
of $b_i$ that is between the points $x_j$ and $f_j(x_j)$ in the
fundamental domain. Now, from ( \ref{F2.11}) we have
\begin{align}
\int_{\bar{b}_j}\omega_{f_i}&=\int_{x_j}^{f_j(x_j)}d\log W_{(f_i(x_1))-(x_1),x_0}(x)\nonumber\\
&=\log\frac{ W_{(f_i(x_1))-(x_1),x_0}(f_j(x_j))}{W_{(f_i(x_1))-(x_1),x_0}(x_j)}\nonumber\\
&=\log\left(\prod_{h\in
F_0}\frac{w_{(Jf_i(x_1))-(Jx_1)}(Jhf_j(x_j))}{w_{(Jf_i(x_1))-(Jx_1)}(Jh(x_0))}\Big/
\prod_{h\in F_0}\frac{w_{(Jf_i(x_1))-(Jx_1)}(Jh(x_j))}{w_{(Jf_i(x_1))-(Jx_1)}(Jh(x_0))}\right)\nonumber\\
&=\log\mu_{(f_i(x_1))-(x_1)}(f_j)\nonumber\\
&=\sum_{h\in C(f_j|f_i)}\log\frac{w_{(J(z^+(f_j)))-(J(z^-(f_j)))}(Jh(z^+(f_i)))}{w_{(J(z^+(f_j)))-(J(z^-(f_j)))}(Jh(z^-(f_i)))}\nonumber\\
&=\sum_{h\in C(f_j|f_i)}\log \langle
J(z^+(f_j)),J(z^-(f_j)),J(h(z^+(f_i)),J(h(z^-(f_i))\rangle
\end{align}
Where the last equality comes from the definition in chapter one.
Similarly we can reach to the second part of b) using (\ref{F2.12}).
\end{proof}
\subsection{Differentials of the third kind and Green's functions}
For $a,b,c,d\in S_0$ lets denote by the same words the corresponding
points in $P\subset H^2$ the fundamental domain of $F$ ( then
$Ja,Jb,Jc$ and $Jd$ are in $\hat{\mathbb{C}}\backslash
\cup_{i=1}^p(D_i\cup\bar{D}_{p+i})$ the fundamental domain of
$\Gamma$) and put $\nu_{(a)-(b)}=d\log W_{(a)-(b),x_0}(x)$. Then
$\nu_{(a)-(b)}$ is a differential of third kind on  $S_0$ with the
residues 1 and -1 at the images of $a$ and $b$. Also, because the
points $a$ and $b$ are in the fundamental domain then they are out
of the circles $C_i$. Then $\bar{a}_i$ - periods of $\nu_{(a)-(b)}$
are zero, and from (\ref{F2.13}) the  $\bar{b}_k$ - periods are
\begin{align}
\int_{\bar{b}_j}\nu_{(a)-(b)}&=\int_{x_j}^{f_j(x_j)}d\log W_{(a)-(b),x_0}(x)\nonumber\\
&=\log\frac{ W_{(a)-(b),x_0}(f_j(x_j))}{W_{(a)-(b),x_0}(x_j)}\nonumber\\
&=\log\left(\prod_{h\in
F_0}\frac{w_{(J(a))-(J(b))}(Jhf_j(x_j))}{w_{(J(a))-(J(b))}(Jh(x_0))}\Big/
\prod_{h\in F_0}\frac{w_{(J(a))-(J(b))}(Jh(x_j))}{w_{(J(a))-(J(b))}(Jh(x_0))}\right)\nonumber\\
&=\log\mu_{(a)-(b)}(f_j)\nonumber\\
&=\log\prod_{h\in S(f_j)}\frac{w_{(J(a))-(J(b))}(J(z^+(h)))}{w_{(J(a))-(J(b))}(J(z^-(h)))}\nonumber\\
&=\sum_{h\in S(f_j)}\log \langle
J(a),J(b),J(z^+(h)),J(z^-(h))\rangle.
\end{align}
Then we have
\begin{align}
Re\int_{\bar{b}_j}\nu_{(a)-(b)}=\log|\mu_{(a)-(b)}(f_j)|=\sum_{h\in
S(f_j)}\log |\langle J(a),J(b),J(z^+(h)),J(z^-(h))\rangle|.\nonumber
\end{align}
Now, we can reach to a differential of the third kind with pure
imaginary periods by defining
\begin{align}
\omega_{(a)-(b)}=\nu_{(a)-(b)}-\sum_{i=1}^pX_i(a,b)\omega_{f_i}
\end{align}
Where the real coefficients  $X_i(a, b)$ are such that the set of
the equations
\begin{align}
\sum_{j=1}^pX_j(a,b)Re \tau_{ij}=Re \int_{\bar
b_i}\nu_{(a)-(b)}=\sum_{h\in S(f_i)}\log |\langle Ja, Jb, J(z^+(h)),
J(z^-(h))\rangle|\nonumber
\end{align}
for $i= 1,...,p$ are satisfied. In fact the coefficients $X_i(a, b)$
kill the real part of the $\bar{b}_i$ - periods of $\nu_{(a)-(b)}$.
Notice that the new differential form $\omega_{(a)-(b)}$ probably
have nonzero  $\bar{a}_i$ - periods, but since the coefficients
$X_i(a, b)$ are real and $\int_{\bar b_i}\omega_{f_i}$ are pure imaginary  then they are pure imaginary too.\\

Finally, if we denote the points $a,b,c$, and $d$ in $S_{0}$ and the
images of these points in the fundamental domain of $F$ by the same
notations, then we have
\begin{align}
\int_d^c\nu_{(a)-(b)}&=\int_d^cd\log W_{(a)-(b),x_0}(x)\nonumber\\
&=\log\frac{ W_{(a)-(b),x_0}(c)}{W_{(a)-(b),x_0}(d)}\nonumber\\
&=\log\left(\prod_{h\in
F_0}\frac{w_{(J(a))-(J(b))}(Jh(c))}{w_{(J(a))-(J(b))}(Jh(x_0))}\Big/
\prod_{h\in F_0}\frac{w_{(J(a))-(J(b))}(Jh(d))}{w_{(J(a))-(J(b))}(Jh(x_0))}\right)\nonumber\\
&=\log\prod_{h\in F_0}\frac{w_{(J(a))-(J(b))}(Jh(c))}{w_{(J(a))-(J(b))}(Jh(d))}\nonumber\\
&=\sum_{h\in F_0}\log\langle J(a),J(b),J(h(c)),J(h(d))\rangle.
\end{align}
And by (\ref{T2.11})
\begin{align}
\int_d^c\omega_{f_i}&=\int_d^cd\log W_{(f_i(x_1))-(x_1),x_0}(x)\nonumber\\
&=\log\frac{W_{(f_i(x_1))-(x_1),x_0}(c)}{W_{(f_i(x_1))-(x_1),x_0}(d)}\nonumber\\
&=\log\frac{\prod_{h\in S(f_i)}\frac{w_{(J(z^+(h)))-(J(z^-(h)))}
(J(c))}{w_{(J(z^+(h)))-(J(z^-(h)))} (J(x_0))}}{\prod_{h\in
S(f_i)}\frac{w_{(J(z^+(h)))-(J(z^-(h)))}
(J(d))}{w_{(J(z^+(h)))-(J(z^-(h)))} (J(x_0))}}\nonumber\\
&=\log\prod_{h\in S(f_i)}\frac{w_{(J(z^+(h)))-(J(z^-(h)))}
(J(c))}{w_{(J(z^+(h)))-(J(z^-(h)))}(J(d))}\nonumber\\
&=\sum_{h\in S(f_i)}\log\langle
J(z^+(h)),J(z^-(h)),J(c),J(d)\rangle.
\end{align}
Then the Green's function on $S_0$ can be computed as following
\begin{align}
g((a)-(b),(c)-(d))&=Re\int_d^c\omega_{(a)-(b)}\nonumber\\
&=Re\int_d^c\nu_{(a)-(b)}-\sum_{i=1}^pX_i(a,b)Re\int_d^c\omega_{f_i}\nonumber\\
&=\sum_{h\in F_0}\log\mid\langle
J(a),J(b),J(h(c)),J(h(d))\rangle\mid-\nonumber\\
&\sum_{i=1}^pX_i(a,b)\sum_{h\in S(f_i)}\log\mid\langle
J(z^+(h)),J(z^-(h)),J(c),J(d)\rangle\mid
\end{align}
Because of the commutativity of the diagram
$$H^2\hspace{5mm}\overset{J}{\longrightarrow}\hspace{5mm}\Omega(\Gamma)$$
$$\pi_F\searrow\hspace{7mm}\swarrow \pi_\Gamma  $$
$$S_0$$
the image of the point $x\in H^2$ and $J(x)\in\Omega(\Gamma)$ are
the same in $S_0$. Then the above formula for the Green's function
on $S_0$ gives an expression via the  points on $S_0$. In fact for
the points $a,b$ in $\Omega(F)$ the images of the geodesics
$\{a,b\}$ and $\{J(a),J(b)\}$ in the Kleinian manifold $M_F$ are the
same. We can see this by using the following covering spaces
\begin{align}
H^3\cup\Omega(F)=H^3\cup
H^2\cup(-H^2)\overset{J}{\longrightarrow}\frac{H^3}{N}\cup
\frac{H^2}{N}\cup\frac
{-H^2}{N}\overset{\pi_\Gamma}{\longrightarrow}M_F.
\end{align}
For all parts $Aut(\tilde{X})\simeq\Gamma$, and again we have
$\pi_\Gamma\circ J=\pi_F$. Also since  $J$ is extended on
$\Lambda(F_0)$ in a natural way then for each $f\in F_0$ the image
of the geodesics $\{Z^+(f),Z^-(f)\}$ and $\{J(Z^+(f)),J(Z^-(f))\}$
in $N_F$ are the same.\\

One can reach to a formula for the Green's function on $S_1$
similarly by replacing $-H^2$ and $-P$ instead of $H^2$ and $P$.
\begin{defn} A finitely generated, torsion free Kleinian group $Q$
is called Quasi - Fuchsian if the limit set $\Lambda(Q)$ be a Jordan
curve and each of the two simply connected components of $\Omega(Q)$
be $Q$-invariant.
\end{defn}
\begin{prop} Given a Quasi - Fuchsian group $Q$, there exist a
Fuchsian group $F$ and a quasiconformal diffeomorphism between
Kleinian manifolds $M_Q$ and $M_F$.
\end{prop}
\begin{proof} ( See \cite{A}).\end{proof}
Lets denote by $D_1$ and $D_2$  the  simply connected components of
$\Omega(Q)$ for Quasi-Fuchsian group $Q$. Then there are two
Fuchsian groups $F_1$ and $F_2$ related to $Q$ such that $D_i/Q$ is
homeomorphic to  $H^2/F_i$. Then each $F_i$ is isomorphic to $Q$.
Also since $Q|_{\Lambda(Q)}$ is topologically conjugate to
$F_i|_{\hat{\mathbb{R}}}$ ( or $F_i|_{S^1}$ in unit disk model )
then naturally there is a Markov map
$f_Q:\Lambda(Q)\rightarrow\Lambda(Q)$ like Fuchsian case. Then all
the statements for Fuchsian groups can be extended to the Quasi -
Fuchsian groups. In this case because the Jordan curve $\Lambda(Q)$
is generally not smooth nor even rectifiable ( \cite{LB} p.263 )
then the hausdorff dimension of the subgroup $Q_0$ of $Q$ ( like
$F_0$ for Fuchsian case $F$) may be not less than 1 in general. We
have the following theorem
\begin{thm} For a finitely generated Kleinian group $G$ the following
conditions  are equivalent\\

\noindent(1) $M_G = (H^3\cup\Omega(G))/G$ is diffeomorphic to
$S\times[0,1]$ where $S$ is a component of $\partial M_G$;\\
(2) $G$ is Quasi - Fuchsian;\\
(3) $\Omega(G)$ has an invariant component that is a Jordan domain;\\
(4) $\Omega(G)$ has two invariant components;
\end{thm}
\begin{proof} ( See \cite{HK}, page 125).\end{proof}
Also two homeomorphic Riemann surfaces can be made uniform
simultaneously by a single Quasi - Fuchsian group ( theorem of Bers
on simultaneous uniformization). Then we have the computations for
all hyperbolic 3 - manifolds with two compact Riemann surfaces with
the same genus as it's boundary components.

\section{$n$ Boundary components case}
Already, we have computed the Green's function for the boundary
components of a hyperbolic 3-manifolds with two boundary at infinity
with the same genera. In this chapter we want to do the problem for
the most general case that the Green's function can be defined i.e.
for manifolds with $n<\infty$ boundary components that are compact
Riemann surfaces probably with different genera. The case with two
boundary is different genera are included in this case too. Such
manifolds are uniformized by some Kleinian groups and for the first,
we will try to find a decomposition for such Kleinian groups. This,
help us to find invariants of these groups that are necessary for
computing  the automorphic functions that will be used for computing
the differentials and the Green's functions. In all of this section
we consider $G$ to be a finitely generated and torsion free Kleinian
group of the second kind ( i.e. $\Omega(G)\neq\varnothing$).

\begin{thm}\textbf{(The Ahlfors finiteness theorem)} Let $G$ be a
finitely generated and torsion free Kleinian group. Then $\partial
M_G = \Omega(G)/G$ is a finite union of analytically finite Riemann
surfaces( closed Riemann surface from which a finite number of
points are removed ).\end{thm}
\begin{proof}( See \cite{A}).\end{proof}

For the rest of the section we consider
\begin{align}
\partial M_G = \frac{\Omega(G)}{G}=S_1\cup S_2\cup ...\cup S_{n+1}
\end{align}
such that for each $i$, $S_i$ is a compact Riemann surface of genus
$p_i$ and without cusped point. Then  we should consider that $G$
doesn't have any parabolic elements. Also consider that the
component $S_{n+1}$ is the one that for $i=1,...,n$, $p_{n+1}\geq
p_i$. In this case $G$ is a function group ( A finitely generated
non - elementary Kleinian group which has an invariant component in
its region of discontinuity ) and there is an invariant component
$\Omega_0$ in $\Omega(G)$.

\begin{prop}\label{P3.11} Let $G$ be a function group  with an invariant
component $\Omega_0$ in $\Omega(G)$. Then
$\Lambda(G)=\partial\Omega_0$. And hence all the other components of
$\Omega(G)$ are simply connected.
\end{prop}
\begin{proof}( See \cite{HK}).\end{proof}
We consider two cases for $\Omega_0$ in the proposition above:\\
\textbf{Case1:} The component $\Omega_0$ is simply connected. Then
in this case $G$ is a B - group ( a function group with simply
connected invariant component ). And we know that a B-group with a
compact boundary component say $S_k$ is a
Quasi - Fuchsian group ( See \cite{GFK} page 411).\\
\textbf{Case2:} The component $\Omega_0$ is not simply connected. In
this case we have the following theorem.

\begin{thm} Let $G$ be a function group with invariant component
$\Omega_0$ that is not a simply connected component. Then $G$ has a
decomposition into a free product of subgroups as following
$$G= \mathfrak{B}_1\ast
\cdots\ast\mathfrak{B}_r\ast\mathfrak{A}_1\ast\cdots\ast\mathfrak{A}_s\ast
<p_1>\ast\cdots\ast <p_t>\ast <f_1>\ast\cdots\ast <f_u>$$ Where each
$\mathfrak{B}_i$ is a B - group , $\mathfrak{A}_i$ is a free abelian
group of rank two with two parabolic generators, $<p_i>$ is the
cyclic group generated by the parabolic transformation $p_i$ and
$<f_i>$ is the cyclic group generated by the loxodromic
transformation $f_i$. Furthermore, each parabolic transformation in
$G$ is conjugate in $G$ to one a listed subgroup. $G$ has finite -
sided fundamental polyhedron if and only if all the groups
$\mathfrak{B}_i$ do.
\end{thm}
\begin{proof} \cite{GFK} page 420 ( see also \cite{HK}).\end{proof}

This theorem shows that in our special case we have
\begin{align}
G = \mathfrak{B}_1\ast\cdots\ast\mathfrak{B}_r\ast
<f_1>\ast\cdots\ast <f_u>.
\end{align}
For B - groups if one boundary component be compact then it's a
Quasi - Fuchsian group. And $<f_1>\ast\cdots\ast <f_u>$ is a free
and purely loxodromic group and consequently it's a Schottky group
\cite{HK}. Hence we have
\begin{align}
G = Q_1\ast\cdots\ast Q_n\ast\Gamma
\end{align}
Where $Q_i$ is a Quasi - Fuchsian group and $\Gamma$ is a schottky
group. Then by Van Kampen theorem we have
\begin{align}
M_G = M_{Q_1}\# M_{Q_2}\#\cdots\# M_{Q_n}\# M_\Gamma
\end{align}
and if $\partial M_{Q_i}=S_{Q_i}\cup S'_{Q_i}$ then
\begin{align}
S_{n+1}=S'_{Q_1}\#S'_{Q_2}\#\cdots\#S'_{Q_n}\#S_{\Gamma}\hspace{5mm}
and \hspace{5mm} S_{Q_i}=S_i.
\end{align}
We will denote by $p_0$ the genus  of Schottky group $\Gamma$ in the
decomposition of the group $G$. We intend to analyze the structure
of the region of discontinuity of $G$. For the first, we consider
the case that we don't have the Schottky group $\Gamma$ in the
decomposition of $G$. Lets denote by $D_i$ and $D_i'$ the simply
connected components of the Quasi - Fuchsian group $Q_i$.

\begin{lem}\label{L3.11} For each $g\in G$, $g(D_i)\cap D_i=D_i$ or
$g(D_i)\cap D_i=\emptyset$.\end{lem}

\begin{lem}\label{L3.12} For each $i=1,...,n$, the sets $\mathcal{D}_i=\bigcup_{g\in
G}g(D_i)$ are $G$ - invariant. Then
$\bigcap_1^n\mathcal{D}_i=\emptyset.$
\end{lem}

From proposition (\ref{P3.11} ) all other components than $\Omega_0$
of $\Omega(G)$ are simply connected. Now, since $S_1\cap\cdots\cap
S_n=\emptyset$  and from lemmas ( \ref{L3.11} ), ( \ref{L3.12} ) and
the decomposition of $\partial M_G$ we see that $\{ D_1,D_2,...,D_n,
\Omega_0\}$ is a complete system of representatives of the
equivalent classes ( $\Omega_1$ and $\Omega_2$ are equivalent or
conjugate if and only if $stab_G(\Omega_1)$ is conjugate in $G$ to
$stab_G(\Omega_2)$; and since $z^\pm (hgh^{-1})= hz^\pm(g)$ then
$\Omega_1$ and $\Omega_2$ are equivalent or conjugate if and only if
there is a $g$ in $G$ such that $\Omega_1=g(\Omega_2)$ ) of the
components of $\Omega(G)$. Since $\Omega_0$ is $G$ - invariant the
class of $\Omega_0$ has one member. Then we have

\begin{prop}\label{P3.15} If we consider that the class of $D_i$ is for
$S_{Q_i}=S_i$, then
\begin{align}
\Omega_0=\mathbb{\hat C}\setminus(\bigcup_{i=1}^n\bigcup_{g\in G}
 g({\bar D}_i))=\bigcap_{i=1}^n\bigcap_{g\in G} g(D_i').\nonumber
\end{align}
\end{prop}
\begin{proof} $\Omega(G)$ contains only the $G$ - invariant sets
$\mathcal{D}_i$ and $\Omega_0$.\end{proof}

 By lemmas ( \ref{L3.11} ), ( \ref{L3.12} ) and proposition
( \ref{P3.15} ) we have
\begin{align}
\Lambda(G)=\bigcup_{i=1}^n \bigcup_{g\in
G}g(\partial(D_i))=\bigcup_{i=1}^n \bigcup_{g\in G}g(\Lambda(Q_i)).
\end{align}
( Figure \ref{Fig3.13} for the case that $n=3$ and there is no
Schottky group ). Then we have
\begin{figure}[!hbtp]
\centerline{\includegraphics{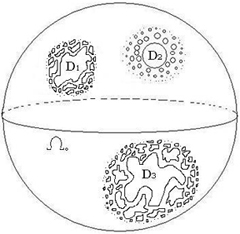}} \caption { }\label{Fig3.13}
\end{figure}
\begin{align}
S_1=D_1/Q_1,\cdots,S_n=D_n/Q_n\hspace{1cm} and
\hspace{1cm}S_{n+1}=\Omega_0/G.
\end{align}

\subsection{Fundamental domain of the group $G$} Let
$K_i=F_i\cup F_i'$ be a fundamental domain for the group $Q_i$ such
that $F_i\subset D_i$ and $F_i'\subset D_i'$. Then a fundamental
domain for the group $G$ is ( Figure \ref{Fig3.14} )
\begin{align}
K=(\bigcup_{i=1}^nF_i)\cup(\bigcap_{i=1}^nF_i').
\end{align}
\begin{figure}[!hbtp]
\centerline{\includegraphics{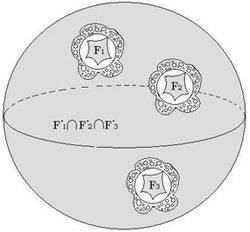}} \caption { }\label{Fig3.14}
\end{figure}
In this case we have $p_{n+1}=p_1+p_2+\cdots+p_n.$\\

Now, lets consider the general case that there is  a Schottky group
$\Gamma_0=< \gamma_{01},..., \gamma_{0 p_0} >$ (of order $p_0$) in
the decomposition of the group $G$. Again we know that all other
components of $\Omega(G)$ than $\Omega_0$ is simply connected. Then
$\partial M_{\Gamma_0}$ can not be glued to the components other
than $S_{n+1}$ in $\partial M_G$. Also this means that in this case
we have
\begin{align}
\Omega_0=(\mathbb{\hat C}\setminus(\bigcup_{i=1}^n\bigcup_{g\in G}
 g({\bar D}_i)))\setminus\bigcup_{g\in
 G}g(\Lambda(\Gamma_0))=\bigcap_{g\in G}(\bigcap_{i=1}^n g(D_i')\setminus
g(\Lambda(\Gamma_0)))\nonumber
\end{align}
 And
\begin{align}
\Lambda(G)=\bigcup_{i=1}^n\bigcup_{g\in G}
g(\Lambda(Q_i)\cup\Lambda(\Gamma_0)).
\end{align}
 In this case if $\{A_1,...,A_{2p_0}; \gamma_{01},...,
\gamma_{0 p_0} \}$ be a marking for the Schottky group $\Gamma_0$,
then the fundamental domain is
\begin{align}
K=(\bigcup_{i=1}^nF_i)\cup((\bigcap_{i=1}^nF_i')\backslash\bigcup_{i=1}^{p_0}(A_i\cup\bar{A}_{p_0+i})).
\end{align}
And also for the genus of the component $S_{n+1}$ we have
$p_{n+1}=p_0+p_1+\cdots+p_n$.

\subsection{Green's Function on $S_i$} \label{S3.10} For each $i=1,...,n$  lets consider
\begin{align}
Q_i=\langle \{q_{ij} ; j=1,...,2p_i \}; \prod_{j=1}^{p_i} [q
_{ij},q_{ip_i+j}]=I \rangle.
\end{align}
Then we can compute Green's function on the components $S_i$ for
$i=1,...,n$ as like as the previous section.

\subsection{Schottky group associated to the group  $G$} From the definition of $Q_i$ we know that
$Q_i/N_i$ is a free group generated by $p_i$  generators. We have
the following lemma

\begin{lem} Lets denote by $N$ the smallest normal subgroup of the
group $Q_1\ast\cdots\ast Q_n$ including the generators $q_{ij}$ for
$i=1,...,n$ and $j=p_i+1,...,2p_i$ and $N_i$ be the smallest normal
subgroup of the group $Q_i$ including the generators $q_{ij}$ for
$j=p_i+1,...,2p_i$.
 Then we have
\begin{align}
\frac{Q_1\ast\cdots\ast
Q_n}{N}\cong\frac{Q_1}{N_1}\ast\cdots\ast\frac{Q_n}{N_n}.
\end{align}
\end{lem}
\begin{proof} In the case $n=2$ the map
\begin{center}
$\varphi : Q_1\ast Q_2\longrightarrow
\frac{Q_1}{N_1}\ast\frac{Q_2}{N_2}$\\
$x_1...x_k\longmapsto \bar{x}_1...\bar{x}_k$
\end{center}
is an isomorphism. The general case is true by induction on $n$.
\end{proof}

 Now if $N$ be the smallest normal subgroup of the group $G$
including the generators $q_{ij}$ for $i=1,...,n$ and
$j=p_i+1,...,2p_i$, Then by the previous lemma one can see that
$G/N$ is isomorphic to the group
$$\Gamma_0\ast\frac{Q_1}{N_1}\ast\cdots\ast\frac{Q_n}{N_n}
=\Gamma_0\ast \Gamma_1\ast\cdots\ast \Gamma_n.$$
Then $G/N$ is a free group of order $p_{n+1}$.\\

Now lets consider the covering space map $\Omega_0\longrightarrow
S_{n+1}$. Since $N$ is a normal subgroup of $G$ then the covering
space $\Omega_0/N\longrightarrow S_{n+1}$ is  regular  and is
between $\Omega_0\longrightarrow S_{n+1}$ with the covering
transformations group $Aut(\Omega_0/N)=G/N$ that is a free group of
order $p_{n+1}$. Then similar to the previous section we have a
Schottky group $\Gamma=\langle
\gamma_{i1},\gamma_{i2},...,\gamma_{ip_i}\quad| \quad
i=0,...,n\quad\rangle$ ( we can arrange the generators of the group
$\Gamma$ in this way because of the isometry ) and a complex -
analytic covering mapping $J:\Omega_0\longrightarrow \Omega(\Gamma)$
such that the following diagram is commutative
\begin{center}
$\Omega_0\hspace{8mm}\overset{J}{\longrightarrow}\hspace{5mm}\Omega(\Gamma)$\\
$\pi_G\searrow\hspace{10mm}\swarrow \pi_\Gamma$\\
$S_{n+1}$
\end{center}
and for $i=1,...,n$ and $j\leq p_i$, $J\circ q_{ij}=\gamma_{ij}\circ
J$ and for $j> p_i$, $J\circ q_{ij}=J$ and $J\circ \gamma_{0j}
=\gamma_{0j}\circ J$ (we have considered  the same notations for the
generators of $\Gamma_i$ and $\Gamma$ ). As a matter of fact there
is a Fuchsian group $F$ and there are normal subgroups $H_1$ and
$H_2$ of $F$ such that the sequence
\begin{align}
H^2\overset{J'}{\longrightarrow}\frac{H^2}{H_1}\cong\Omega_0\overset{J}{\longrightarrow}
\frac{H^2}{H_2}\cong\frac{\Omega_0}{N}\longrightarrow S_{n+1}
\end{align}
is the composition of analytic covering maps. Then we have the
composition of covering maps
\begin{align}
H^2\quad\overset{J\circ
J'}{\longrightarrow}\quad\frac{H^2}{H_2}\cong\frac{\Omega_0}{N}\quad\longrightarrow
S_{n+1}
\end{align}
and since $F/H_2=Aut(\Omega_0/N)=G/N$ is a free group, like the
previous section we have $\Omega_0/N\cong
H^2/H_2\cong\Omega(\Gamma)$ for some Schottky group $\Gamma$. In
fact, for the first, because of the isometries
$F/H_1=Aut(H^2/H_1\cong\Omega_0)=G$ we can choose the Fuchsian group
\begin{align}
F = \langle \{f_{ij}\ ; i=0,\ldots,n\quad j=1,...,2p_i
\};\prod_{i=0}^n\prod_{j=1}^{p_i} [ f_{ij},f_{i,p_i+j}]=I\quad
\rangle
\end{align}
such that the equations $g_{ij}\circ J'=J'\circ f_{ij}$ for
$i=1,...,n$ and $j=1,...,p_i$ are satisfied. Where
$g_{ij}(=q_{ij}\quad or\quad \gamma_{0j})$ are the generators of
$G$. Then like the previous section we can choose uniquely up to
conjugation in $\mathrm{PGL}(2, \mathbb{C})$ the Schottky group
$\Gamma$ that satisfies the equations
\begin{align}
(J\circ J')\circ f_{ij}=\gamma_{ij}\circ(J\circ J')\hspace{5mm} and
\hspace{5mm}  (J\circ J')\circ f_{ip_i+j}=J\circ J'
\end{align}
 Then since $J'$ is onto, the relations $J\circ g_{ij}=\gamma_{ij}\circ J$ and $J\circ
g_{ip_i+j}=J$ are satisfied automatically for $i=1,...,n$ and
$j=1,...,p_i$. This
 means that the following diagram is commutative in all loops
$$H^2\quad\overset{J'}{\longrightarrow}\quad\Omega_0\quad\overset{J}{\longrightarrow}\quad\Omega(\Gamma)$$
$$\quad f_{ij}\updownarrow\hspace{13mm} g_{ij}\updownarrow\hspace{15mm} \gamma_{ij}\updownarrow\hspace{13mm}$$
$$H^2\quad\overset{J'}{\longrightarrow}\quad\Omega_0\quad\overset{J}{\longrightarrow}\quad\Omega(\Gamma)$$
This gives a proof for the existing the of covering space
$\Omega(\Gamma)$ and Schottky group $\Gamma$ satisfying the
properties that we need and also shows the relations between the
Fuchsian group $F$ , Kleinian group $G$ and the Schottky group
$\Gamma$ associated to
$S_{n+1}$.\\

Now lets put $G_0=\Gamma_0\ast Q_{10}\ast Q_{20}\ast\cdots \ast
Q_{n0}$. Where $Q_{i0}$ is the free subgroup of $Q_i$ generated by
the elements $q_{i1},...,q_{ip_i}$. For the extension of $J$ to
$\Lambda(G_0)$  i.e. defining the map
$J:\Omega_0\cup\Lambda(G_0)\longrightarrow\mathbb{\hat{C}}$, we know
that each $Q_{i0}$ is free and purely loxodromic then it is a
Schottky group. Like the previous section we can code an element of
$\Lambda(Q_{i0})$ by Schottky coding $x=\cdots
q_{ij_2}^{\epsilon_{j_2}}q_{ij_1}^{\epsilon_{j_1}}q_{ij_0}^{\epsilon_{j_0}}(x_0)$
for $j_k\leq p_i$ and a point $x_0$ in $\Omega_0\subset D'_i$. For
each element $g$ in $G_0$ lets denote by $\gamma$ the element in
$\Gamma$ corresponding to $g$, such that $J\circ g=\gamma\circ J$.
Now put
$$J_0:\bigcup_{g\in G}g(\Lambda(\Gamma_0))\longrightarrow
\Lambda(\Gamma)$$
$$g(\cdots\gamma_{i_2}^{\epsilon_{2}}\gamma_{i_1}^{\epsilon_{1}}\gamma_{i_0}^{\epsilon_{0}}(x_0))\mapsto
\gamma(\cdots
\gamma_{0i_2}^{\epsilon_{2}}\gamma_{0i_1}^{\epsilon_{1}}\gamma_{0i_0}^{\epsilon_{0}}(z_0))$$
Where $z_0=J(x_0)$. And for $i=1,...,n$ put
$$J_i:\bigcup_{g\in G}g(\Lambda(Q_{i0}))\longrightarrow
\Lambda(\Gamma)$$
$$g(\cdots
q_{ij_2}^{\epsilon_{j_2}}q_{ij_1}^{\epsilon_{j_1}}q_{ij_0}^{\epsilon_{j_0}}(x_0)
)\mapsto \gamma(\cdots
\gamma_{ij_2}^{\epsilon_{j_2}}\gamma_{ij_1}^{\epsilon_{j_1}}\gamma_{ij_0}^{\epsilon_{j_0}}(z_0))$$
 And finally lets define
$$J:\Lambda(G_0)\longrightarrow \Lambda(\Gamma)$$
$$J(x)=\begin{cases}J_0(x)&if~ x\in \bigcup_{g\in
G}g(\Lambda(\Gamma_0))\\~\\J_i(x) &if~x\in\bigcup_{g\in
G}g(\Lambda(Q_{i0}))\end{cases}.$$\\
Then like the previous section we
can show that for each element $g$ in $G_0$  and $\gamma$  the
element of $\Gamma$ corresponding to $g$ and each $x$ in
$\Lambda(G_0)$ we have $J\circ g(x)=\gamma\circ J(x)$ and
$J(z^\pm(g))=z^\pm(\gamma)$. Also like the previous chapter  we can
consider the Fuchsian coding for the points of $\Lambda(G_0)$ and
from (\ref{F2.9}) we can explain $J$ and consequently the Green's
function on $S_{n+1}$ via the Fuchsian coding.

\subsection{Green's function of $S_{n+1}$} If we consider the condition $a(\Gamma)<1$ and use the
subgroup $G_0=\Gamma_0\ast Q_{10}\ast Q_{20}\ast\cdots \ast Q_{n0}$
instead of $F_0$ in the previous section then one can bring all of
the definitions like before with some changes, and compute the
Green's function on $S_{n+1}$ and the other parts in the same way.
In this case we should notice that if $\{a_{ij},b_{ij} |i=0,...,n ,
j=1,...,p_i\}$ be a set that makes a base for $H_1(S_{n+1},
\mathbb{Z})$ then each member of it has a class of images in
$\Omega_0$. But we can consider that the representatives that are
used in the computations are in the fundamental domain. As its shown
in figure \ref{Fig3.14}. In this case these images are in
$D_i'\cap\Omega_0$ for $i=1,...,n$ and also in
$\Omega(\Gamma_0)\cap\Omega_0$ for $i=0$. Then the representatives
for $S_i$ and $S_{n+1}$ are coincide if we uniformize $S_i$ by
$D'_i$, i.e. identifying both components of $\partial M_{Q_i}$.
Final formula for the Green's function on $S_{n+1}$ is as following
\begin{align}
g_{s_{n+1}}((a)-(b)&,(c)-(d))=\sum_{f\in G_0}\log\mid\langle
J(a),J(b),J(f(c)),J(f(d))\rangle\mid\nonumber\\
&-\sum_{\substack{i=0,...,n\\
j=1,...,p_i}}\overline{X}_{ij}(a,b)\sum_{f\in
S(f_{ij})}\log\mid\langle
J(z^+(f)),J(z^-(f)),J(c),J(d)\rangle\mid\label{F3.10}
\end{align}

Where $f_{ij}$ is equal to $q_{ij}$ for $i=1,...,n$ and to
$\gamma_{0j}$ for $i=0$, and $\overline{X}_{ij}(a,b)$ are the
multipliers for $g_{s_{n+1}}$.

\subsection{Remark (Infinitely many boundary component case )}
When we have infinitely many boundary components in the boundary of
$N_G$ that are Riemann surfaces without puncture, according to the
Ahlfors finiteness theorem, $G$ is an infinitely generated Kleinian
group. And hence one of the boundary components of $M_G$ is with
infinity genus and is not a compact Riemann surface then for this
component the Green's function is not defined. But for the other
components we can bring the computations similar to the previous
condition using some subgroups of $G$ that are isomorphic to the
Kleinian groups in the previous sections.

\end{document}